\documentclass[10pt]{article}
\pdfoutput=1

\font\smallit=cmti10
\font\smalltt=cmtt10

\usepackage{amssymb,amsmath,epsfig,amsthm,mathrsfs,tikz,color,hyperref,xcolor}

\hypersetup{
    colorlinks,
    linkcolor={red!50!black},
    citecolor={blue!50!black},
    urlcolor={blue!80!black}
}
\usetikzlibrary{arrows}
\usetikzlibrary{decorations.markings}

\makeatletter

\newcommand{\rest}[2]{\left[{#1}\right]_{#2}}

\newcommand{\q}{\emptyset}
\newcommand{\n}{\mathscr{N}}
\renewcommand{\o}{\mathscr{O}}
\newcommand{\p}{\mathscr{P}}
\newcommand{\pn}{\mathscr{P}\!\mathscr{N}}
\newcommand{\no}{\mathscr{N}\!\mathscr{O}}
\newcommand{\op}{\mathscr{O}\!\mathscr{P}}
\newcommand{\any}{\mathscr{N}\!\mathscr{O}\!\mathscr{P}}
\newcommand{\nim}{\textsc{Nim }}
\newcommand{\bphi}{\boldsymbol{\Phi}}
\newcommand{\bpi}{\boldsymbol{\Pi}}

\renewcommand\section{\@startsection {section}{1}{\z@}
{-30pt \@plus -1ex \@minus -.2ex}
{2.3ex \@plus.2ex}
{\normalfont\normalsize\bfseries}}

\renewcommand\subsection{\@startsection{subsection}{2}{\z@}
{-3.25ex\@plus -1ex \@minus -.2ex}
{1.5ex \@plus .2ex}
{\normalfont\normalsize\bfseries}}

\renewcommand{\@seccntformat}[1]{\csname the#1\endcsname. }

\makeatother
\newtheorem{theorem}{Theorem}
\newtheorem{lemma}[theorem]{Lemma}

\newtheorem{proposition}[theorem]{Proposition}
\newtheorem{corollary}[theorem]{Corollary}

\newtheorem{question}{Question}

\newtheoremstyle{mythm}%
{3pt}
{3pt}
{}
{}
{\bfseries}
{.}
{.5em}
{}%

\theoremstyle{mythm}
 \newtheorem{defn}[theorem]{Definition}

\numberwithin{theorem}{section}

\begin{document}

\begin{center}
\uppercase{\bf An extension of the normal play convention to $N$-player combinatorial games}
\vskip 20pt
{\bf Mark Spindler}\\
{\smallit The Johns Hopkins Center for Talented Youth, Baltimore, Maryland 21209, USA}\\
{\tt mark.edward.spindler@gmail.com}
\end{center}
\vskip 20pt

\phantom{\centerline{\smallit Received: , Revised: , Accepted: , Published: }}
\vskip 30pt

\centerline{\bf Abstract}

\noindent
We examine short combinatorial games for three or more players under a new play convention in which a player who cannot move on their turn is the unique loser.

We show that many theorems of impartial and partizan two-player games under normal play have natural analogues in this setting. For impartial games with three players, we investigate the possible outcomes of a sum in detail, and determine the outcomes and structure of three-player \textsc{Nim}.

\pagestyle{myheadings} 
\markright{\smalltt \protect\phantom{INTEGERS: 20A (2020)}\hfill}
\thispagestyle{empty} 
\baselineskip=12.875pt 
\vskip 30pt

\section{Introduction}
The vast majority of prior work in combinatorial game theory has focused on two-player games. In this paper, we investigate a novel play convention for $N$-player games. When no move is available, the next player is declared the unique loser of the game, and the other players all win equally. This convention was independently conceived by Salt in \cite{mse}. In one sense, this complicates matters, since more than one player may have a winning strategy. However, as compared to previous work on multiplayer games, this convention leads to results that better parallel those for two-player games under the normal play convention. Our work and objects of study fit into the general framework for combinatorial game theory laid out in \cite{wastlund}.

A summary of much of the prior work done on combinatorial games with three or more players can be found in I.4 of \cite{cgt}. This paper is closely connected with the related work of Propp, Cincotti, and Doles (n\'{e}e Greene).

In \cite{propp}, Propp considers three-player impartial games with a certain play convention: In analogy with two-player normal play, when no move is available, the previous player is declared the unique winner of the game. This yields a recursive characterization of outcome classes $\mathscr N$, $\mathscr O$, and $\mathscr P$ in which the test for $\mathscr O$ contains a proviso reminiscent of the mis\`{e}re play convention for two-player games. 

In \cite{cinc}, Cincotti considers three-player partizan games under a convention which reduces to Propp's in the impartial case. When a player cannot move on their turn, they are eliminated. The other players continue to play from that position. The last player remaining is the unique winner. Cincotti extends this convention to $N$-player games in \cite{ncinc}.

In \cite{doles}, Doles (n\'{e}e Greene) considers three-player partizan games with a convention similar to Propp's. The last player to move wins first place. The player before that wins second place. And the player who cannot move on their turn loses/is ranked third place. Under this convention, an assumption that players always play rationally leads to some player having a winning strategy. 

In this paper, we consider the play convention in which the first player without an available move is the unique loser of the game. In Section~\ref{sec:3imp}, we examine three-player impartial games and give special attention to three-player \textsc{Nim}, closely paralleling the work of Propp in \cite{propp}. We examine how the outcome of a disjunctive sum can vary depending on the outcomes of the summands, discuss reverting games, construct some undetermined games which make all sums undetermined, and derive the finite quotient describing play in three-player \textsc{Nim}. In Section~\ref{sec:nimp}, we generalize much of the content of the previous section to $N$ players. We do not completely characterize the outcomes of a sum or of \textsc{Nim}, but some partial results are obtained, including the outcomes of all \nim positions with two heaps. In Section~\ref{sec:part}, we consider partizan games with $N$ players. We define a preorder based on favorability to a particular player in disjunctive sums, characterize all comparisons with the empty game $0$ (and show that $0$ is not equal to any other partizan game), show that the simplification theorems regarding dominated and reversible options have natural analogues for $N$ players, and examine some particular games related to the two-player integer games. We conclude in Section~\ref{sec:open} with a collection of open questions.

\section{Three-Player Impartial Games}\label{sec:3imp}
\subsection{Preliminaries}
In this subsection, we introduce the games under discussion, the notation and terminology we use for three-player impartial games, and some initial observations.
\subsubsection{Notation}
Unless otherwise specified, all games in this paper are assumed to be short games; they have finitely many distinct subpositions, and admit no infinite runs. 

As is common, an impartial game is usually considered to be the set of its options, $+$ denotes the disjunctive sum of games, $n\cdot G$ denotes the sum of $n$ copies of $G$, $G'$ and $G^*$ typically denote an option of a game $G$, etc. We use $\cong$ to indicate that two games are isomorphic (i.e. the game trees are isomorphic). Additionally, we also make use of structural induction in proofs throughout, often without comment.

Many impartial games in this paper are be built out of nim-heaps. $*n$ denotes a nim-heap of size $n$. Equivalently, $*n\cong\left\{*0,*1,\ldots,*(n-1)\right\}$. As special cases, $0$ denotes the game with no options $\{\}$, and $*$ denotes $*1$ (so that $*\cong\{0\}$). 

We define the \textit{normal} play convention for three-player games to be the one in which the next player to move is the unique loser of the empty game $0$. The other two players both win equally.

\begin{defn}\label{def:3impout}
The (normal-play) \textit{outcome} $o(G)$ of a game $G$ is a subset of $\left\{\mathbf{N},\mathbf{O},\mathbf{P}\right\}$ whose members are determined recursively as follows.
\begin{itemize}
	\item $\mathbf{N}\in o(G)$ exactly when there exists an option $G^*$ with $\mathbf{P}\in o\left(G^*\right)$.
	\item $\mathbf{O}\in o(G)$ exactly when every option $G'$ satisfies $\mathbf{N}\in o(G')$.
	\item $\mathbf{P}\in o(G)$ exactly when every option $G'$ satisfies $\mathbf{O}\in o(G')$.
\end{itemize}
\end{defn}

By an induction argument, the elements of $o(G)$ tell us who has a winning strategy in $G$. 
$\mathbf{N}$ represents the next player to move, $\mathbf{P}$ being the previous player, and $\mathbf{O}$ being the other player. When referring to players in the discussion of a single position, we often use ``Next'', ``Other'', and ``Previous'' to refer to the players. 

For brevity, we use $\q$, $\n$, $\o$, $\p$, $\pn$, $\no$, $\op$, and $\any$ to denote the corresponding subsets of $\{\mathbf N,\mathbf O,\mathbf P\}$. For example, $\no=\left\{\mathbf{N},\mathbf{O}\right\}$. 

Unlike in the two-player case, the outcome of a game need not correspond to the set of winners when all players play perfectly. (This is also noted in \cite{doles} and \cite{propp}.)

\begin{proposition}
In $G\cong\{0,\{*,*+*\}\}$, Next can only guarantee that they win by winning with Previous. However, Next and Other can also collude to ensure that Previous loses.
\end{proposition}
\begin{proof}
If Next plays perfectly, then they must move to $0$. Otherwise, Other could choose to move to $*$ on their turn. However, if Next does not move to $0$, then Other can choose to move to $*+*$ to force (the original) Previous to lose. Therefore, $o(G)=\n$ despite the fact that both Next and Previous would win if Next were to play perfectly. \end{proof}

\subsubsection{Observations}\label{sssec:obs}

\begin{proposition}$o(G)$ is always a proper subset of $\any$.\end{proposition}
\begin{proof}Let $G$ be a game with $o(G)=\any$. Then, by the recursive Definition~\ref{def:3impout}, we can choose $G^{*}$ with $o(G^{*})=\any$. Since games do not admit any infinite runs, no such $G$ can exist.\end{proof}

\begin{proposition}\label{prop:sharps cycle}
$\mathbf{N}\in o(G)$ exactly when $\mathbf{O}\in o\left(\{G\}\right)$, $\mathbf{O}\in o(G)$ exactly when $\mathbf{P}\in o\left(\{G\}\right)$, and $\mathbf{P}\in o(G)$ exactly when $\mathbf{N}\in o\left(\{G\}\right)$.
\end{proposition}

\begin{proof}
Immediate from the recursive definitions of Definition~\ref{def:3impout}.
\end{proof}

\begin{proposition}\label{prop:7outs}All seven proper subsets of $\any$ are possible outcomes.\end{proposition}\label{prop:impouts}
\begin{proof}These examples are readily verified: $o(0)=\op$, $o(*)=\pn$, $o\left(\{*\}\right)=\no$, $o(*2)=\n$, $o\left(\{*2\}\right)=\o$, $o\left(\{\{*2\}\}\right)=\p$, and $o\left(\{*2,\{*2\}\}\right)=\q$.\end{proof}

\begin{defn}\label{def:undet}
A game $G$ is said to be \textit{undetermined} if $o(G)=\q$.  
\end{defn}

\begin{proposition}\label{prop:Q dooms}
If $G$ has an undetermined option $G^{*}$ and has no option $G'$ with $\mathbf{P}\in o(G')$, then $G$ is undetermined.
\end{proposition}
\begin{proof}
The existence of $G^{*}$ means that $\mathbf{O},\mathbf{P}\notin o(G)$, and the other condition forces $\mathbf{N}\notin o(G)$.
\end{proof}

\subsection{General Results}\label{ssec:3imp}
In this subsection, we investigate the outcomes of sums of games, and derive some equations and inequations for games.

\subsubsection{Outcomes of Sums}
In the two-player theory of impartial games under normal play, if we put aside the more-precise Sprague\textendash Grundy Theory, there are two main facts worth noting about the outcomes of sums as compared to the outcomes of the summands: 1. the outcome of $G+G$, and 2. the possible outcomes of $G+H$\textemdash if $o(H)=\p$ then $o(G+H)=o(G)$, but if $o(G)=o(H)=\n$ then $o(G+H)$ can be either $\p$ or $\n$. We examine three-player analogues of these results.

\begin{theorem}[Next Generation]\label{thm:ngenerates}
If $\mathbf{N}\notin o(H)$, then $o(G+H)\subseteq o(G)$.
\end{theorem}
In other words, adding on a component to a game cannot introduce a winning strategy for a new player unless Next has a winning strategy in the new component. Since $o(0)=\op$, this is similar to the two-player result that adding a $\p$-position to a game does not change its outcome.
\begin{proof}

Let $G,H$ be games, and suppose $\mathbf{N}\notin o(H)$. 

If $\mathbf{P}\notin o(G)$, then we can choose an option $G^{*}$ with $\mathbf{O}\notin o\left(G^{*}\right)$. But then $\mathbf{O}\notin o\left(G^{*}+H\right)$ by inductive hypothesis, so that $\mathbf{P}\notin o(G+H)$.

If $\mathbf{O}\notin o(G)$, then we can choose an option $G^{*}$ with $\mathbf{N}\notin o\left(G^{*}\right)$. But then $\mathbf{N}\notin o\left(G^{*}+H\right)$ by induction, so that $\mathbf{O}\notin o(G+H)$.

If $\mathbf{N}\notin o(G)$, then $\mathbf{P}\notin o\left(G'\right)$ for all options $G'$. Similarly, $\mathbf{P}\notin o\left(H'\right)$ for all options $H'$. Therefore, by induction, we have both $\mathbf{P}\notin o\left(G'+H\right)$ and $\mathbf{P}\notin o\left(G+H'\right)$ for arbitrary options. As such, $\mathbf{N}\notin o(G+H)$ by definition.
\end{proof}

\begin{corollary}\label{cor:sumundet}If $G$ and $H$ are undetermined, then $G+H$ is undetermined.\end{corollary}
\begin{proof}By \hyperref[thm:ngenerates]{Next Generation}, $\mathbf N\notin o(H)$ yields $o(G+H)\subseteq o(G)=\q$.\end{proof}

\hyperref[thm:ngenerates]{Next Generation} has an important consequence for sums of two games.
\begin{theorem}[Other Procreation]\label{thm:pgeneration}
If $\mathbf O\notin o(G),o(H)$, then $\mathbf P\notin o(G+H)$.
\end{theorem}
In other words, a sum of two games cannot produce a winning strategy for Previous unless Other has a winning strategy in one of the summands. 
\begin{proof}
Since $\mathbf O\notin o(G),o(H)$, we can choose options $G^{*}$ and $H^{*}$ such that $\mathbf N\notin o\left(G^{*}\right),o\left(H^{*}\right)$. By \hyperref[thm:ngenerates]{Next Generation}, $\mathbf N\notin o\left(G^{*}+H^{*}\right)$, so that $\mathbf O\notin o\left(G+H^{*}\right)$ and $\mathbf P\notin o(G+H)$.
\end{proof}

With \hyperref[thm:pgeneration]{Other Procreation} in hand, we can strengthen Corollary~\ref{cor:sumundet}.
\begin{corollary}\label{cor:psbuynothing}
If $o(G),o(H)\subseteq\p$, then $G+H$ is undetermined.
\end{corollary}
\begin{proof}By \hyperref[thm:ngenerates]{Next Generation}, $o(G+H)\subseteq o(G)\subseteq \p$. But by \hyperref[thm:pgeneration]{Other Procreation}, $\mathbf P\notin o(G+H)$. Therefore, $o(G+H)=\q$.
\end{proof}
\begin{corollary}\label{cor:qgetsnop}
If $G$ is undetermined, then $\mathbf{P}\notin o(G+H)$ for any $H$.
\end{corollary}
\begin{proof}
Suppose $(G,H)$ is a counterexample. By \hyperref[thm:ngenerates]{Next Generation}, $\mathbf{N}\in o(H)$ as otherwise $o(G+H)\subseteq o(G)=\q$. Also, since $\mathbf{N}\notin o(G)$, we have $o(G+H)\subseteq o(H)$ so that $\mathbf{P}\in o(H)$ as well. Since $o(H)\ne\any$, we have $\mathbf{O}\notin o(H)$, and the contradiction follows immediately from \hyperref[thm:pgeneration]{Other Procreation}.
\end{proof}

With two players, $G+H$ can have any normal-play outcome not disallowed by the two-player version of \hyperref[thm:ngenerates]{Next Generation} (e.g. if $o(G)=o(H)=\n$ then $o(G+H)$ can be either $\p$ or $\n$). This result has a three-player analogue.

\begin{theorem}\label{thm:allpossible}
All outcomes of a sum of two games not immediately disallowed by \hyperref[thm:ngenerates]{Next Generation} or \hyperref[thm:pgeneration]{Other Procreation} are possible.
\end{theorem}
\begin{proof}
For a list of representative examples, see Table~\ref{tbl:allsums} in the \hyperref[sec:appendix]{Appendix}. For instance, $o(*2)=\n$, $o(*)=\pn$, and $o(*2+*)=\no$. 
\end{proof}
The possible outcomes of sums are summarized in Table~\ref{tbl:addition}, where entries indicate that any \emph{subset} (other than $\any$) is possible. For example, if $o(G)=\o$ and $o(H)=\pn$, then $o(G+H)$ may be any of $\q,\n,\p,\pn$.

\begin{table}[h!]
\centering
\resizebox{\linewidth}{!}{%
\begin{tabular}{r|lllllll}
$+$   & $\q$ & $\n$ & $\o$ & $\p$ & $\no$ & $\op$ & $\pn$  \\ \hline
$\q$  & $\q$ & $\n$ & $\q$ & $\q$ & $\no$ & $\q$ & $\n$ \\
$\n$  & $\n$ & $\no$ & $\n$ & $\n$ & $\any$ & $\n$ & $\no$ \\
$\o$  & $\q$ & $\n$ & $\o$ & $\q$ & $\no$ & $\o$ & $\pn$ \\
$\p$  & $\q$ & $\n$ & $\q$ & $\q$ & $\no$ & $\p$ & $\n$ \\
$\no$ & $\no$ & $\any$ & $\no$ & $\no$ & $\any$ & $\no$ & $\any$ \\
$\op$ & $\q$ & $\n$ & $\o$ & $\p$ & $\no$ & $\op$ & $\pn$ \\
$\pn$ & $\n$ & $\no$ & $\pn$ & $\n$ & $\any$ & $\pn$ & $\no$ 
\end{tabular}
}
\caption{The three-player pairwise addition table}
\label{tbl:addition}
\end{table}

As in \cite{propp}, we also examine sums of a game with itself. For example, if $o(G)=\q$ or $o(G)=\p$, then $o(G+G)=\q$ by Corollary~\ref{cor:psbuynothing}. And if $o(G)=\o$, then $o(G+G)=\o$ or $o(G+G)=\q$,  by \hyperref[thm:ngenerates]{Next Generation}. In fact, there are no new obstructions for $o(G+G)$.

\begin{theorem}
$G+G$ can have any outcome not disallowed by \hyperref[thm:ngenerates]{Next Generation} and \hyperref[thm:pgeneration]{Other Procreation}.\end{theorem}
\begin{proof}
For examples of all $23$ cases, see Table~\ref{tbl:doublesums} in the \hyperref[sec:appendix]{Appendix}. For instance, $o(*2)=\n$ and $o(*2+*2)=\o$.
\end{proof}

The possible outcomes of the sum of a game with itself are summarized in Table~\ref{tbl:doubling}.

\begin{table}[h!]
\centering
\begin{tabular}{r|l}
$o(G)$   & $o(G+G)$ \\ \hline
$\q$  & $=\q$ \\
$\n$  & $\subseteq\no$ \\
$\o$  & $\subseteq\o$ \\
$\p$  & $=\q$ \\
$\no$ & $\subseteq\any$ \\
$\op$ & $\subseteq\op$ \\
$\pn$ & $\subseteq\no$ 
\end{tabular}
\caption{The three-player doubling table}
\label{tbl:doubling}
\end{table}

In the impartial two-player setting, the Tweedledum\textendash Tweedledee strategy yields $o(G+G)=\p$. A similar result holds for three players.
\begin{proposition}\label{prop:tripleno}
For all games $G$, $\mathbf N\notin o(G+G+G)$.
\end{proposition}
\begin{proof}
The other two players can guarantee Next's loss by mirroring their moves in the other two components.
\end{proof}

Combined with \hyperref[thm:ngenerates]{Next Generation} and \hyperref[thm:pgeneration]{Other Procreation}, this theorem reduces the possibilities for $o(G+G+G)$ to those listed in Table~\ref{tbl:trebling}. 

\begin{table}[h!]
\centering
\begin{tabular}{r|l}
$o(G)$   & $o(G+G+G)$ \\ \hline
$\q$  & $=\q$ \\
$\n$  & $\subseteq\op$\\
$\o$  & $\subseteq\o$ \\
$\p$  & $=\q$ \\
$\no$ & $\subseteq\op$ \\
$\op$ & $\subseteq\op$ \\
$\pn$ & $\subseteq\op$ 
\end{tabular}
\caption{The three-player trebling table}
\label{tbl:trebling}
\end{table}

\begin{proposition}\label{prop:trebling}
All cases in Table~\ref{tbl:trebling} are possible, except perhaps for $o(G)=\n$, where it is not known whether $\mathbf P\in o(G+G+G)$ occurs.
\end{proposition}
\begin{proof}
See Table~\ref{tbl:treblesums} in the \hyperref[sec:appendix]{Appendix} for example games covering the known cases. For instance, $o(*)=\pn$, but $o(*+*+*)=\op$.
\end{proof} 

\subsubsection{Equality of Impartial Games}

As in the two-player case, we are interested in when games are similar enough to be replaced in disjunctive sums. As is standard, we define $=$ by the fundamental equivalence (see \cite{cgt}).
\begin{defn} (Impartial Equality)
We say $G$ \textit{equals} $H$ and write $G=H$ if $o(G+X)=o(H+X)$ for all impartial games $X$.
\end{defn}

We first examine reversible moves, as they are key to understanding equality classes of two-player games, both in normal and mis\`{e}re play. Our presentation is based on the similar Definition~V.1.3 in \cite{cgt}.

\begin{defn}
Let $G,H$ be games. We say that $H$ is \textit{revertible to} $G$ if both of the following hold.
\begin{itemize}
\item For every option $G'$ of $G$, $H$ has a corresponding option $\widehat{G'}$ equal to $G'$.
\item For every option $H'$ of $H$ not equal to some $G'$, $H'$ has a second option ${{H'}^{**}}$ equal to $G$.
\end{itemize} 
\end{defn}

In two-player settings, if $H$ is revertible to $G$ (with a slight modification in mis\`{e}re play), then  $H=G$. With three-players under normal play, only one conclusion holds in general (see Proposition~\ref{prop:threestar}).

\begin{theorem}\label{thm:reverting}
If $H$ is revertible to $G$, then $o(H+X)\subseteq o(G+X)$ for all $X$. 
\end{theorem}
\begin{proof}

Suppose $\mathbf{N}\in o(H+X)$. If the winning move has the form $H+X^{*}$, then $\mathbf{P}\in o(H+X^{*})$ so that $\mathbf{P}\in o(G+X^{*})$ by induction. Hence, the same move would win in $G+X$, and $\mathbf{N}\in o(G+X)$. Now suppose the winning move in $H+X$ is to something of the form $\widehat{G^*}+X$ with $\widehat{G^*}$ equal to an option ${G^*}$ of $G$. Then $\mathbf{P}\in o(\widehat{G^*}+X)=o(G^*+X)$, so that $\mathbf{N}\in o(G+X)$. Finally, if the winning move is to an option $H^*$ not equal to an option of $G$, then the other two players can move to yield some ${H}^{***}+X$ where ${H}^{***}=G$. Note that $\mathbf{N}\in o({H}^{***}+X)$ as we started with a winning move in $H+X$. From ${H}^{***}=G$ it follows that $\mathbf{N}\in o(G+X)$, as desired.

Suppose $\mathbf{O}\in o(H+X)$. Then $\mathbf{N}\in o(H+X')$ for all options $X'$. By induction, $\mathbf{N}\in o(G+X')$ for all $X'$ as well. Since $\mathbf{O}\in o(H+X)$, we also have $\mathbf{N}\in o(\widehat{G'}+X)=o(G'+X)$ for all options $G'$ as well. We have shown $\mathbf{N}\in o\left((G+X)'\right)$ for all options, so $\mathbf{O}\in o(G+X)$. 

The argument for $\mathbf{P}\in o(H+X)$ is exactly analogous to that for $\mathbf{O}$.
\end{proof}

The phrasing of Corollary~\ref{cor:replacement} is from the Replacement Lemma~V.1.2 in \cite{cgt}.
\begin{corollary}\label{cor:replacement} 
Suppose that $G\cong \left\{G_i':\,i\in I \right\}$, $H\cong \left\{H_i':\,i\in I \right\}$, and $G_i'=H_i'$ for all $i$. Then $G=H$.
\end{corollary}
\begin{proof}
By symmetry, it suffices to show that $o(H+X)\subseteq o(G+X)$ for all $X$. But that follows immediately from Theorem~\ref{thm:reverting} since $H$ is revertible to $G$. 
\end{proof}

Unfortunately, the reverse inclusion for Theorem~\ref{thm:reverting} does not always hold in this three-player setting. For instance, it is not true that $3\cdot*=0$, even though in the two-player setting, $2\cdot*=0$ holds in the both the partizan normal-play and impartial mis\`{e}re-play contexts.

\begin{proposition}\label{prop:threestar}
$3\cdot*\ne0$.
\end{proposition}
\begin{proof}
It is useful to give the players consistent names throughout, even though which one is next to play varies; for this purpose, we call the first three players to move ``Nicky'', ``Olly'', and ``Pat''.

Define $X\cong\{\{0,*+*\}\}$. In $0+X$, Olly can win by making the move to $0$. But in the sum $3\cdot*+X$, Nicky and Pat can stop Olly from winning. First, Nicky moves in the $3\cdot*$ component to $*+*$. Then, no matter which move Olly makes, it's possible for Pat to move to a position isomorphic to $\{3\cdot*,*,\{0,*+*\}\}$, from which Nicky can move to $3\cdot*$. The lines of play are illustrated in the game graph below.
\end{proof}

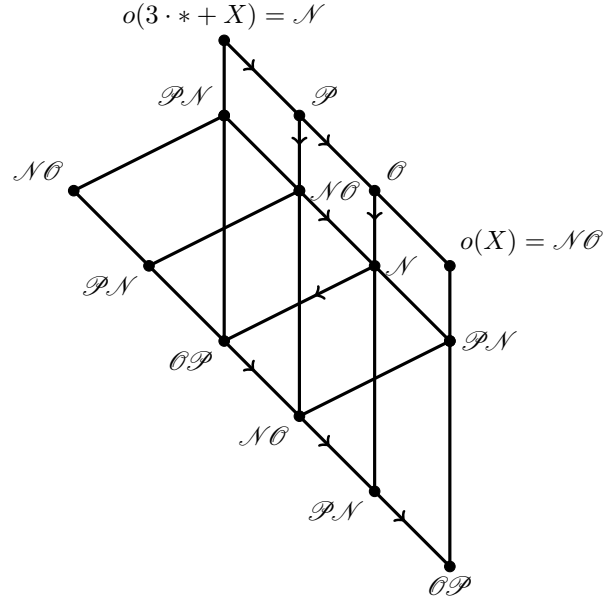
\begin{center}
\begin{tikzpicture}
\begin{scope}[very thick,decoration={
    markings,
    mark=at position 0.4 with {\arrow{>}}}
    ] 
\draw [-] (0.,3.) -- (0.,0.) node [below] {$\op$}; 
\draw [-] (0.,4.) -- (0.,3.) node [right] {$\pn$};
\draw [-] (-1.,4.) -- (0.,3.);
\draw [-] (-1.,4.) -- (-1.,1.) node [below left] {$\pn$};
\draw [-] (0.,3.) -- (-2.,2.) node [below left] {$\no$};
\draw[postaction={decorate}] [-] (-2.,2.) -- (-1.,1.);
\draw[postaction={decorate}] [-] (-1.,1.) -- (0.,0.);
\draw[postaction={decorate}] [-] (-3.,3.) -- (-2.,2.);
\draw[postaction={decorate}] [-] (-1.,4.) -- (-3.,3.) node [below left] {$\op$};
\draw [-] (-4.,4.) -- (-3.,3.);
\draw[postaction={decorate}] [-] (-2.,5.) -- (-1.,4.) node [right] {$\n$};
\draw [-] (-2.,5.) -- (-4.,4.) node [below left] {$\pn$};
\draw [-] (-2.,5.) -- (-2.,2.);
\draw[postaction={decorate}] [-] (-1.,5.) -- (-1.,4.);
\draw [-] (-1.,5.) -- (0.,4.) node [above right] {$o(X)=\no$};
\draw [-] (-5.,5.) -- (-4.,4.);
\draw [-] (-3.,6.) -- (-3.,3.);
\draw [-] (-3.,6.) -- (-2.,5.) node [right] {$\no$};
\draw [-] (-3.,6.) -- (-5.,5.) node [above left] {$\no$};
\draw[postaction={decorate}] [-] (-2.,6.) -- (-2.,5.);
\draw[postaction={decorate}] [-] (-2.,6.) -- (-1.,5.) node [above right] {$\o$};
\draw [-] (-3.,7.) -- (-3.,6.) node [above left] {$\pn$};
\draw[postaction={decorate}] [-] (-3.,7.) node [above] {$o(3\cdot*+X)=\n$} -- (-2.,6.) node [above right] {$\p$};
\end{scope}
\foreach \x in {0,1,2,3,4,5}
{\draw[fill] (-\x,\x) circle [radius=2pt];}
\foreach \x in {0,2,3}
{\draw[fill] (-\x,\x+3) circle [radius=2pt];}
\foreach \x in {0,1,2,3}
{\draw[fill] (-\x,\x+3) circle [radius=2pt];
\draw[fill] (-\x,\x+4) circle [radius=2pt];}
\end{tikzpicture}
\end{center}

\begin{defn}\label{def:absorb}A game $G$ is said to be \textit{absorbing} if $G+H$ is undetermined for any game $H$.
\end{defn}

\begin{proposition}\label{prop:absorbingunique}
If $G$ and $H$ are absorbing games, then $G=H$.
\end{proposition}
\begin{proof}Note that $o(G+X)=\q=o(H+X)$ for all $X$.\end{proof}
Therefore, there is at most one equality class of absorbing games. 
\begin{proposition}\label{prop:blackholeabsorbing}

If $G$ is an absorbing game, then $G+H=G$ for any $H$.
\end{proposition}
\begin{proof}Note that $o\left((G+H)+X\right)=o\left(G+(H+X)\right)=\q=o(G+X)$.\end{proof}

Proposition~\ref{prop:blackholeabsorbing} justifies the name ``absorbing''.

\begin{theorem}[Absorbing Game Construction]\label{thm:blackholegen}
Suppose that $G$ has at least one option, and all options $G'$ are undetermined ($o(G')=\q$). Then $G$ is an absorbing game.
\end{theorem}
\begin{proof}

Let $G$ be as in the statement. Since $o(G')=\q$ for all options $G'$, $\mathbf{P}\notin o(G'+X)$ by Corollary~\ref{cor:qgetsnop}. By induction, $o(G+X')=\q$ for all $X'$. If $X\ncong0$, then $o(G+X)=\q$ by Proposition~\ref{prop:Q dooms} since $o\left(G+X^*\right)=\q$ for some option $X^*$. If $X\cong0$, then since $G$ has at least one option $G^*$ with $o\left(G^*\right)=\q$, $o(G+X)=o(G)=\q$ by Proposition~\ref{prop:Q dooms}.
\end{proof}

\begin{corollary}An absorbing game exists.\end{corollary}
\begin{proof}By Proposition~\ref{prop:7outs}, $o\left(\left\{*2,\{*2\}\right\}\right)=\q$. Therefore, by the \hyperref[thm:blackholegen]{Absorbing Game Construction}, $\left\{\left\{*2,\{*2\}\right\}\right\}$ is an absorbing game.\end{proof}

The \nim position $3\cdot*3$ is also absorbing (see Lemma~\ref{lem:333}). By Proposition~\ref{prop:absorbingunique}, all such games are equal.

\begin{proposition}\label{prop:3undetstrong}If $G,H,J$ are undetermined, then all options of $G+H+J$ are undetermined.
\end{proposition}
\begin{proof}
Let $G,H,J$ be undetermined. Suppose, for sake of contradiction, that $G+H+J$ has an option that is not undetermined. Without loss of generality, suppose $J^*$ is an option of $J$ satisfying $o\left(G+H+J^*\right)\ne\q$. Since $G$ is undetermined, Corollary~\ref{cor:qgetsnop} yields $\mathbf{P}\notin o\left(G+(H+J^*)\right)$. Therefore, either $\mathbf{N}\in o\left(G+H+J^*\right)$ or $\mathbf{O}\in o\left(G+H+J^*\right)$.

Suppose $\mathbf{N}\in o\left(G+H+J^*\right)$. Then either $\mathbf{P}\in o\left(G+H+J^{**}\right)$ for some $J^{**}$ or (without loss of generality) $\mathbf{P}\in o\left(G+H^{*}+J^{*}\right)$ for some $H^*$. But both are impossible by Corollary~\ref{cor:qgetsnop}.

Now suppose $\mathbf{O}\in o\left(G+H+J^*\right)$. Since $G,H$ are undetermined, Corollary~\ref{cor:psbuynothing} yields that $G+H$ is undetermined. By \hyperref[thm:ngenerates]{Next Generation}, $\mathbf{N}\in o\left(J^*\right)$. Hence, we can choose a winning move $J^{**}$. Then $\mathbf{N}\in o\left(G+H+J^{**}\right)$. By the argument from the previous paragraph, Corollary~\ref{cor:qgetsnop} yields a contradiction in this case, too.
\end{proof}

\begin{corollary}\label{cor:3undetabs}If $G$ is undetermined, then $3\cdot G$ is an absorbing game.\end{corollary}
\begin{proof}Combine Proposition~\ref{prop:3undetstrong} with the  \hyperref[thm:blackholegen]{Absorbing Game Construction}.\end{proof}

Note that we can find many equal absorbing games with Proposition~\ref{prop:blackholeabsorbing} and Corollary~\ref{cor:3undetabs}. And by using Corollary~\ref{cor:replacement}, we can build many games equal to a given one that has an absorbing subposition.

\subsection{Three-Player Nim}\label{ssec:3nim}

\nim is a classic ruleset in Combinatorial Game Theory that can serve as a demonstrative case study. In this subsection, we calculate the outcomes of all \nim positions, as well as the quotient corresponding to playing only \textsc{Nim}. When playing \nim with three players under normal play, there are only finitely many different classes of positions to be concerned with, though more than in the play convention considered in \cite{propp}.

In the language of \cite{universes}, the positions of three-player \nim form a three-player analogue of a ``universe''. But it is not ``parental'' as all \nim positions are impartial. And it is not ``dense'' even in an impartial sense, as no \nim position has outcome $\p$ (see Theorem~\ref{thm:allnimoutcomes}).

\subsubsection{Outcomes of Nim Positions}
We classify the outcomes of \nim positions. We start by noting the outcomes of $40$ particular positions, and then use a few lemmas to justify that the apparent patterns extend to all positions. This approach is similar to applications of the Octal Periodicity Theorem (see Thm. IV.2.7 of \cite{cgt}). 

\begin{lemma}\label{lem:40nim} 
The outcomes of 40 \nim positions of the form $k\cdot*+G$ 
 are recorded in Table~\ref{tbl:40nim}.
\begin{table}[h!]
\centering
\begin{tabular}{r|ccccc}
$G\backslash k$ & $0$ & $1$ & $2$ & $3$\\
\hline
$0$ & $\op$ & $\pn$ & $\no$ & $\op$\\
$*2$ & $\n$ & $\no$ & $\pn$ & $\n$\\
$*3$ & $\n$ & $\no$ & $\n$ & $\n$\\
$*4$ & $\n$ & $\no$ & $\n$ & $\n$\\
$*2+*2$ & $\o$ & $\n$ & $\no$ & $\o$\\
$*2+*3$ & $\q$ & $\n$ & $\no$ & $\q$\\
$*2+*4$ & $\q$ & $\n$ & $\no$ & $\q$\\
$*2+*2+*2$ & $\q$ & $\q$ & $\q$ & $\q$\\
$*2+*2+*3$ & $\q$ & $\q$ & $\q$ & $\q$\\
$*2+*2+*4$ & $\q$ & $\q$ & $\q$ & $\q$\\
\end{tabular}
\caption{The three-player outcomes of $40$ \nim positions}
\label{tbl:40nim}
\end{table}
\end{lemma}
\begin{proof}These outcomes can be calculated manually or with a computer program.\end{proof}

\begin{defn}
A 
game $G$ is said to be \textit{3-periodic} if the infinite sequence of outcomes $\left(o(G),o(*+G),o(2\cdot*+G),\dots\right)$ has period $3$.\end{defn}

\begin{lemma}[Nim Periodicity]\label{lem:3periodicity}
Let $G$ be a game. Suppose that all options $G'$ are 3-periodic and that $o\left(G\right)=o\left(3\cdot*+G\right)$. Then $G$ is 3-periodic as well.
\end{lemma}
\begin{proof}
Suppose $k\ge4$. Then the options of $k\cdot*+G$ are $\left(k-1\right)\cdot*+G$ and those of the form $k\cdot*+G'$ for options $G'$. These have the same outcomes as $\left(k-4\right)\cdot*+G$ (by induction on $k$) and $\left(k-3\right)\cdot*+G'$ (since $G'$ is 3-periodic), respectively. But those are exactly all of the options of $\left(k-3\right)\cdot*+G$. Since $k\cdot*+G$ has the same outcomes of options, $o\left(k\cdot*+G\right)=o\left(\left(k-3\right)\cdot*+G\right)$, as desired.
\end{proof}

\begin{lemma}[Nim Stability]\label{lem:3stability}
Let $G$ be a game. Suppose that for all subpositions $H$ of $G$ (including $G$ itself) $o\left(H+*4\right)=o\left(H+*3\right)$. Then $o\left(G+*m\right)=o\left(G+*3\right)$ for all $m\ge3$.
\end{lemma}
\begin{proof}
Let $m\ge4$. Via induction on $m$, it suffices to verify that $o\left(G+*\left(m+1\right)\right)=o\left(G+*m\right)$.

The options of $G+*\left(m+1\right)$ are $G+*m$, those of the form $G+*\ell$ for $\ell<m$ (all of which are options of $G+*m$), and those of the form $G'+*\left(m+1\right)$. 

By induction on $m$, $o\left(G+*m\right)=o\left(G+*\left(m-1\right)\right)$. And by induction on $G$, $o\left(G'+*\left(m+1\right)\right)=o\left(G'+*m\right)$. Trivially, $o\left(G+*\ell\right)=o\left(G+*\ell\right)$. In all cases, an option of $G+*\left(m+1\right)$ has the same outcome as an option of $G+*m$.

Conversely, the options of $G+*m$ are those of the form $G+*\ell$ for $\ell<m$, and those of the form $G'+*m$. All of which have outcome equal to that of an option of $G+*\left(m+1\right)$. 

As outcomes of options are the same, $o\left(G+*(m+1)\right)=o(G+*m)$. 
\end{proof}

\begin{lemma}\label{lem:nimtable2} 
The patterns of outcomes for the \nim positions of the form $k\cdot*+G$ in Table~\ref{tbl:smallnim} are accurate.\end{lemma}
\begin{proof}
Apply \hyperref[lem:3periodicity]{Nim Periodicity} to the upper $7$ rows of Table~\ref{tbl:40nim}, and then apply \hyperref[lem:3stability]{Nim Stability} to $0$, followed by $*2$.\end{proof}
\begin{table}[h!]
\centering
\begin{tabular}{r|ccc}
$G\backslash k\pmod3$ & $\equiv0$ & $\equiv1$ & $\equiv2$\\
\hline
$0$ & $\op$ & $\pn$ & $\no$\\
$*2$ & $\n$ & $\no$ & $\pn$\\
$*m$ & $\n$ & $\no$ & $\n$\\
$*2+*2$ & $\o$ & $\n$ & $\no$\\
$*2+*m$ & $\q$ & $\n$ & $\no$
\end{tabular}
\caption{The outcomes of some \nim positions of the form $k\cdot*+G$ ($m\ge3$)}
\label{tbl:smallnim}
\end{table}

\begin{lemma}\label{lem:33}
For all $k\ge0$, and $m_2\ge m_1\ge3$, $o\left(k\cdot*+*m_1+*m_2\right)=\q$.\end{lemma}
\begin{proof}
Set $G\cong k\cdot*+*m_1+*m_2$.
If $k=0$, then $G$ has the undetermined option $*2+*m_2$ by Lemma~\ref{lem:nimtable2}.
If $k\ge1$, then $G$ has the undetermined option $(k-1)\cdot*+*m_1+*m_2$ by induction on $k$. 
By Proposition~\ref{prop:Q dooms}, it suffices to show that no option $G'$ has $\mathbf{P}\in o(G')$.
By induction and Lemma~\ref{lem:nimtable2}, the only subpositions with that property are of the form $j\cdot*+*2$ or $j\cdot*$, neither of which is possible for an option.
\end{proof}

\begin{lemma}\label{lem:22n} 
For all $k\ge 0$ and $n\ge 2$, $o(k\cdot*+*2+*2+*n)=\q$.\end{lemma}
\begin{proof}
Apply \hyperref[lem:3periodicity]{Nim Periodicity} to the lower $3$ rows of Table~\ref{tbl:40nim}, and then apply \hyperref[lem:3stability]{Nim Stability} to $*2+*2$.\end{proof}

\begin{lemma}\label{lem:233}
For all $k\ge0$, and $m_2\ge m_1\ge3$, $o\left(k\cdot*+*2+*m_1+*m_2\right)=\q$.\end{lemma}
\begin{proof}
By Lemma~\ref{lem:33}, the option $k\cdot*+*m_1+*m_2$ is undetermined. The outcomes of options $G'$ are listed by Lemmas \ref{lem:nimtable2} through \ref{lem:22n}. And the only subpositions where Previous has a winning strategy are of the form $j\cdot*+*2$ or $j\cdot*$. By Proposition~\ref{prop:Q dooms}, $k\cdot*+*2+*m_1+*m_2$ is undetermined.
\end{proof}

\begin{lemma}\label{lem:333}
For all $m_1,m_2,m_3\ge 3$, $*m_1+*m_2+*m_3$ is an absorbing game.

\end{lemma}
\begin{proof}
By the \hyperref[thm:blackholegen]{Absorbing Game Construction} (Theorem~\ref{thm:blackholegen}), it suffices to show that all options have outcome $\q$. If all heaps of an option have size at least $3$, then the outcome is $\q$ by induction. If an option has reduced a heap to $*2$, then the outcome is $\q$ by Lemma~\ref{lem:233}. If the move reduces a heap to $*$ or eliminates a heap, then the outcome is $\q$ by Lemma~\ref{lem:33}.
\end{proof}

\begin{lemma}\label{lem:2222}
For all $n_2\ge n_1\ge2$, $*2+*2+*n_1+*n_2$ is an absorbing game.
\end{lemma}
\begin{proof}
By the \hyperref[thm:blackholegen]{Absorbing Game Construction}, it suffices to show that all options have outcome $\q$. If an option has four heaps of size at least $2$ then the outcome is $\q$ by induction. Otherwise, the outcome is $\q$ by Lemma~\ref{lem:22n} or \ref{lem:233}, as applicable.
\end{proof}

\begin{lemma}\label{lem:nimbig}
Any \nim position that has at least three heaps of size at least $3$, or at least four heaps of size at least $2$ is an absorbing game.
\end{lemma}
\begin{proof}
By Lemmas~\ref{lem:333} and \ref{lem:2222}, any such position has an absorbing game as a summand. By Proposition~\ref{prop:blackholeabsorbing}, the entire position is an absorbing game.
\end{proof}

\begin{theorem}\label{thm:allnimoutcomes}
The outcome of a \nim  position is determined by the numbers of heaps of size $1$, size $2$, and size at least $3$. Specifically, the outcome is $\q$ unless the position has the form $k\cdot*+*n$ or $k\cdot*+*2+*n$, in which case the outcome is listed in Table~\ref{tbl:smallnim}.
\end{theorem}
\begin{proof}
This follows immediately from Lemma~\ref{lem:40nim} and Lemmas \ref{lem:nimtable2} 
 through \ref{lem:nimbig}.
\end{proof}

\subsubsection{Nim Quotient}
In direct analogy to mis\`{e}re quotients (as covered in V.4 of \cite{cgt}), we define an equivalence relation on the set of \nim positions which is coarser than ``$=$''. While the two-player mis\`{e}re quotient of \nim is infinite, the three-player normal play quotient of \nim has size $16$. 

\begin{defn}
If $G$ and $H$ are \nim positions, then we write $G\equiv H$ if $o(G+X)=o(H+X)$ for all \nim positions $X$.
\end{defn}

\begin{proposition}\label{prop:addingpreservesquotient}
If $G,H,J$ are \nim positions and $G\equiv H$, then $G+J\equiv H+J$.
\end{proposition}
\begin{proof}
Since \nim   positions are closed under addition, $J+X$ is a \nim  position whenever $X$ is, so that $o\left((G+J)+X\right)=o\left(G+(J+X)\right)=o\left(H+(J+X)\right)=o\left((H+J)+X\right)$, where the middle equation comes from $G\equiv H$.
\end{proof}

\begin{proposition}\label{prop:quotientsemigroup}
If $G,H,J,K$ are \nim positions and both $G\equiv H$ and $J\equiv K$ hold, then $G+J\equiv H+K$.
\end{proposition}
\begin{proof}
By two applications of Proposition~\ref{prop:addingpreservesquotient}, $G+J\equiv H+J\cong J+H\equiv K+H\cong H+K$. 
\end{proof}

For convenience, let $\mathcal{Q}$ be the set of equivalence classes of \nim positions under $\equiv$. Note that since $0$ is a \nim position, Proposition~\ref{prop:quotientsemigroup} gives $\mathcal{Q}$ the structure of a commutative monoid. As with the convention for mis\`{e}re quotients, let $\bphi(G)$ denote the equivalence class of a \nim position $G$. We follow convention to write elements of $\mathcal{Q}$ with lowercase letters and the operation multiplicatively. For example, if $x=\bphi(G)$ and $y=\bphi(H)$, then $xy=\bphi(G+H)$; we also write $1=\bphi(0)$.

\begin{proposition}\label{prop:quotientoutcome}
For a \nim position $G$, $\bphi(G)$ determines $o(G)$.
\end{proposition}
\begin{proof}
Since $0$ is a \nim position, note that $G\equiv H$ implies $o(G)=o(G+0)=o(H+0)=o(H)$.
\end{proof}

Since the restriction of $o$ to the \nim positions factors through $\mathcal{Q}$, we adopt notation from V.6 of \cite{cgt} and use $\bpi$ to denote the corresponding map on $\mathcal{Q}$. For example, $\bpi(1)=\bpi\left(\bphi(0)\right)=\op$.

\begin{defn}
The \textit{Nim quotient} is the structure $(\mathcal{Q},\cdot,\bpi)$.
\end{defn}

For convenience, for the remainder of this subsection, define $a=\bphi(*)$, $b=\bphi(*2)$, and $c=\bphi(*3)$. 

\begin{theorem}\label{thm:nimquotient}
The Nim quotient is given by the commutative monoid with presentation $(\mathcal{Q},\cdot)\cong\left\langle a,b,c\mid a^3=1,b^4=b^3=b^2c=bc^2=c^2=c^3=ac^2=a^2c^2 \right\rangle$ (which has order 16) 
and the values of $\bpi$ given in Table~\ref{tbl:bpi}.\end{theorem}

\setlength{\tabcolsep}{2.63pt}
\begin{table}[h!]
\centering
\begin{tabular}{r|cccccccccccccccc}
$x$&$1$&$a$&$a^2$&$b$&$ab$&$a^2b$&$c$&$ac$&$a^2c$&$b^2$&$ab^2$&$a^2b^2$&$bc$&$abc$&$a^2bc$&$c^2$ \\ \hline
$\bpi(x)$&$\op$&$\pn$&$\no$&$\n$&$\no$&$\pn$&$\n$&$\no$&$\n$&$\o$&$\n$&$\no$&$\q$&$\n$&$\no$&$\q$\\
\end{tabular}
\caption{Outcomes of the three-player Nim quotient} 
\label{tbl:bpi}
\end{table}
\setlength{\tabcolsep}{6pt}

To prove Theorem~\ref{thm:nimquotient}, we verify various facts individually as a series of lemmas.

\begin{lemma}\label{lem:16outcomes}
The outcomes in Table~\ref{tbl:bpi} are accurate.
\end{lemma}
\begin{proof}
All of the outcomes follow directly from Lemma~\ref{lem:nimtable2}, with the exception of $\bpi(c^2)$, which follows from Lemma~\ref{lem:33}.
\end{proof}

\begin{lemma}\label{lem:a cubed}
$a^3=1$, that is, $*+*+*\equiv0$.
\end{lemma}
\begin{proof}
In Theorem~\ref{thm:allnimoutcomes}, the outcomes only depend on the number of heaps of size $1$ modulo $3$.
\end{proof}

\begin{lemma}\label{lem:distinct elements}
The 16 equivalence classes listed in Table~\ref{tbl:bpi} are distinct.
\end{lemma}
\begin{proof}
Note that $\bpi$ is well defined on equivalence classes, so the only pairs of classes that might be identical are those with the same outcome. We apply Lemma~\ref{lem:a cubed} repeatedly below.

\begin{itemize}
\item $a\ne a^2b$ because $\bpi\left(a^2\right)=\no$ and $\bpi\left(a^3b\right)=\bpi\left(1b\right)=\bpi\left(b\right)=\n$.
\item $a^2,ab,ac,a^2b^2,a^2bc$ are all distinct because the outcomes after multiplying by $a$ are $\op,\pn,\n,\o,\q$, respectively.
\item $b,c,a^2c,ab^2,abc$ are all distinct because the outcomes after multiplying by $a^2$ are $\pn,\n,\no,\o,\q$, respectively.
\item $bc\ne c^2$ because $\bpi(abc)=\n$ and $\bpi(ac^2)=\q$ by Lemma~\ref{lem:33}.
\end{itemize}
\end{proof}

It remains to prove enough other relations to conclude that Table~\ref{tbl:bpi} is complete.

\begin{lemma}\label{lem:gtthree}
For each $n\ge3$, $c=\bphi(*n)$, that is, $*3\equiv *n$.
\end{lemma}
\begin{proof}
In Theorem~\ref{thm:allnimoutcomes}, none of the outcome determinations depend on the exact size of a heap that has size at least $3$.
\end{proof}

\begin{lemma}\label{lem:quothole}
For any $x\in \mathcal{Q}$, $xb^3=xb^2c=xc^2=c^2$.
\end{lemma}
\begin{proof}
We prove this equivalence by showing that all representatives have outcome $\q$. By Lemma~\ref{lem:gtthree}, the only values of $x$ to be concerned about are products of $a$, $b$, and $c$. If $x$ is a power of $a$, then $\bpi(xb^3)=\q$ and $\bpi(xb^2c)=\q$ by Lemma~\ref{lem:22n}, and $\bpi(xc^2)=\q$ by Lemma~\ref{lem:33}. If $x$ contains at least one factor of $b$, then $\bpi(xb^3)=\bpi(xb^2c)=\q$ by Lemma~\ref{lem:2222} and $\bpi(xc^2)=\q$ by Lemma~\ref{lem:233} (if $x$ contains only one factor of $b$ and none of $c$) or Lemma~\ref{lem:nimbig} (if $x$ contains a factor of $c$ or a factor of $b^2$). If $x$ contains a factor of $c$, then $\bpi(xb^3)=\bpi(xb^2c)=\bpi(xc^2)=\q$ by Lemma~\ref{lem:nimbig}. 
\end{proof}

By Lemma~\ref{lem:nimbig}, $*2+*2+*2+*2$ and $*3+*3+*3$ are absorbing games. But note that Lemma~\ref{lem:quothole} tells us that in the context of \nim positions only,  $*2+*2+*2$ and $*3+*3$ have a similar absorbing property.

\begin{corollary}\label{cor:big relation}
$b^4=b^3=b^2c=bc^2=c^2=c^3=ac^2=a^2c^2$.
\end{corollary}

\begin{lemma}
There are only 16 elements of $\mathcal{Q}$ as listed in Table~\ref{tbl:bpi}, and the presentation in the statement of Theorem~\ref{thm:nimquotient} is accurate.
\end{lemma}
\begin{proof}
By $a^3=1$ (from Lemma~\ref{lem:a cubed}) and $b^3=b^4$ and $c^2=c^3$ (from Corollary~\ref{cor:big relation}) and the commutativity of disjunctive sum, the only possible elements of $\mathcal{Q}$ are of the form $a^xb^yc^z$ for $x,z<3,y<4$. However, if $z=2$ or $y+z\ge3$ then the element is equal to $c^2$ by Corollary~\ref{cor:big relation}. This leaves 15 other possibilities with $y+z<3$ and $z<2$, for a total of 16.
\end{proof}

With this, we have proven Theorem~\ref{thm:nimquotient}. For comparison, in Section 5 of \cite{propp}, \nim is analyzed under Propp's play convention in which Previous is the unique winner of $0$. Rephrasing his results in our terminology, the three-player \nim quotient for that play convention has size $9$, though he shows that there are infinitely many equality classes of \nim positions.

Note that in the case of \nim positions alone, we recover something like a more faithful version of \hyperref[thm:ngenerates]{Next Generation} (Theorem~\ref{thm:ngenerates}). 
\begin{corollary}
If $G$ is a \nim position satisfying $o(G)=\op=o(0)$, then $o\left(G+H\right)=o(H)$ for all \nim positions $H$.\end{corollary}
\begin{proof} By Theorem~\ref{thm:nimquotient}, $\bphi(G)=1$. Thus, $o(G+H)=\bpi\left(1\bphi(H)\right)
=o(H)$.\end{proof}

\section{\texorpdfstring{$N$}{N}-Player Impartial Games}\label{sec:nimp}
Many results from Section~\ref{sec:3imp} can be generalized to $N$ players. Throughout this entire section, $N$ refers to the number of players ($N\ge2$). We explicitly mention when assuming $N>2$ is necessary.

For some results, the results and proofs are nearly identical to those in the three-player case, but others become more complex, such as the construction of absorbing games for $N$ players. 

A game is still a finite set of options, and we continue to use $\cong$ to indicate that two games are isomorphic (i.e. the game trees are isomorphic).
\subsection{Preliminaries}
In the impartial setting, we denote the various players as $\mathbf{N},\mathbf{O}_{1},\ldots,\mathbf{O}_{N-2},\mathbf{P}$. $\mathbf{N}$ is the ``Next'' player, $\mathbf{P}$ is the ``Previous'' player, and the rest are the ``Other'' players. At times, it is convenient to use $\mathbf{O}_{0}$ for $\mathbf{N}$ and $\mathbf{O}_{N-1}$ for $\mathbf{P}$.

\begin{defn}\label{def:nplayimp}
The $N$-player (normal-play) \textit{outcome} $o(G)$ of a game $G$ is a subset of $\{ \mathbf{N},\ldots,\mathbf{P}\}$ whose members are determined recursively as follows.
\begin{itemize}
	\item $\mathbf{N}\in o(G)$ exactly when there exists an option $G'$ with $\mathbf{P}\in o(G')$.
	\item For $1\le i\le N-1$, $\mathbf{O}_i\in o(G)$ exactly when every option $G'$ has $\mathbf{O}_{i-1}\in o(G')$.
\end{itemize}
\end{defn}

\begin{proposition}\label{prop:nimpouts}
The following analogues of the propositions in \ref{sssec:obs} all hold.
\begin{enumerate}
\item $o(G)$ is always a proper subset of $\{ \mathbf{N},\mathbf{O}_{1},\ldots,\mathbf{O}_{N-2},\mathbf{P}\}$.
\item For $0\le i\le N-2$, $\mathbf{O}_i\in o(G)$ exactly when $\mathbf{O}_{i+1}\in o\left(\{G\}\right)$. And $\mathbf{P}\in o(G)$ exactly when $\mathbf{N}\in o\left(\{G\}\right)$.
\item If $G$ has an option $G^{*}$ with $o(G^{*})=\emptyset$ and has no option $G'$ with $\mathbf{P}\in o(G')$, then $o(G)=\emptyset$.
\item If $N>2$, then all $2^N-1$ proper subsets are possible outcomes.
\end{enumerate}
\end{proposition}
\begin{proof}
The proofs of claims 1, 2, and 3 are exactly analogous to the proofs in the three-player case in \ref{sssec:obs}. For claim~4, we need a more general idea.

For any non-empty outcome set that includes $\mathbf N$, we can use a game with various options of the form $m\cdot*$. For example, if there are at least $6$ players (if $N\ge6$), then $\{*,3\cdot*,4\cdot*\}$ is a game in which Next chooses whether $\mathbf{O}_{2}$, $\mathbf{O}_{4}$, or $\mathbf{O}_{5}$ will lose; all other players are guaranteed to win. 

For a non-empty outcome set that \emph{does not} include $\mathbf N$, we can apply claim~2 to cycle the outcome set of a game whose outcome includes $\mathbf N$. For example, if $N=6$, then $\left\{\{*,3\cdot*,4\cdot*\}\right\}$ has outcome $\{\mathbf O_1,\mathbf O_2,\mathbf O_4\}$.

Finally, we construct a game $G$ with $o(G)=\emptyset$, if $N>2$, we can take $G\cong\left\{H,\{H\}\right\}$ where $H\cong\{0,\ldots,(N-2)\cdot*\}$. In this game, Next chooses whether $\mathbf O_1$ or $\mathbf O_2$ selects any player (other than themselves) to lose the game.
\end{proof}

\subsection{Outcomes of Sums}

\begin{theorem}[$N$-Player Next Generation]\label{thm:ngeneratesnplayer}
If $\mathbf{N}\notin o(H)$ then $o(G+H)\subseteq o(G)$.
\end{theorem}

Note that $\mathbf{N}\notin o(H)$ is still equivalent to $o(H)\subseteq o(0)$ in the $N$-player setting. In fact, the proof for $N$ players is nearly identical to the one for three players.
\begin{proof}
Let $G,H$ be games, and suppose $\mathbf{N}\notin o(H)$. 

If $\mathbf{O}_{i}\notin o(G)$ for some $i$ with $1\le i\le N-1$, then we can choose an option $G^{*}$ with $\mathbf{O}_{i-1}\notin o\left(G^{*}\right)$. But then $\mathbf{O}_{i-1}\notin o\left(G^{*}+H\right)$ by inductive hypothesis, so that $\mathbf{O}_{i}\notin o(G+H)$.

If $\mathbf{N}\notin o(G)$, then $\mathbf{P}\notin o\left(G'\right)$ for all options $G'$. Similarly, $\mathbf{P}\notin o\left(H'\right)$ for all options $H'$. Therefore, by induction, we have both $\mathbf{P}\notin o\left(G'+H\right)$ and $\mathbf{P}\notin o\left(G+H'\right)$ for arbitrary options. As such, $\mathbf{N}\notin o(G+H)$ by definition.
\end{proof}

We generalize \hyperref[thm:pgeneration]{Other Procreation} in two respects\textemdash more summands, and different players.

\begin{theorem}
For $k\ge1$ and $m\ge0$, if $N>km$ and $\mathbf O_m\notin o(G_1),\ldots o(G_k)$, then $\mathbf O_{km}\notin o(G_1+\cdots+G_k)$.
\end{theorem}
\begin{proof}
Note that the case of $k=1$ is trivial, so we assume $k\ge2$.

First, we consider the case $m=0$, so that $\mathbf O_m=\mathbf N$. By repeatedly applying \hyperref[thm:ngeneratesnplayer]{Next Generation}, $\mathbf N\notin o\left(G_1+\cdots +G_k\right)$.

For $m>0$, we induct on $m$. Since $\mathbf O_m\notin o(G_1),\ldots o(G_k)$, we can choose options $G_i^{*}$ such that $\mathbf O_{m-1}\notin o\left(G_i^{*}\right)$ for all $i$. Then by induction, $\mathbf O_{km-k}\notin o\left(G_1^{*}+\cdots +G_k^{*}\right)$, so that $\mathbf O_{km-(k-1)}\notin o\left(G_1+G_2^{*}+\cdots +G_k^{*}\right)$ and similarly for each successive term until we reach $\mathbf O_{km}\notin o\left(G_1+\cdots +G_k\right)$.
\end{proof}

Proposition~\ref{prop:tripleno} about the outcome of $G+G+G$ generalizes straightforwardly.
\begin{proposition}\label{prop:triplenon}
For all games $G$, $\mathbf N\notin o(N\cdot G)$.
\end{proposition}
\begin{proof}
All players can mirror the moves made by Next in the $N$ components.
\end{proof}

In the $N$-player setting, we can still examine reversible moves.
\begin{defn}
We say that $H$ is \textit{revertible to} $G$ if both of the following hold.
\begin{itemize}
\item For every option $G'$ of $G$, $H$ has a corresponding option $\widehat{G'}$ equal to $G'$.
\item For every option $H'$ of $H$ not equal to some $G'$, $H'$ has an $(N-1)^\text{st}$ option ${{H'}^{*\dots*}}$ equal to $G$.
\end{itemize} 
\end{defn}

\begin{theorem}\label{thm:reverting-n}
If $H$ is revertible to $G$ then $o(H+X)\subseteq o(G+X)$ for all $X$. 
\end{theorem}
The proof is nearly identical to the argument for three players (see Theorem~\ref{thm:reverting}), but is written here in full, for convenience.
\begin{proof}
Suppose $\mathbf{N}\in o(H+X)$. If the winning move has the form $H+X^{*}$, then $\mathbf{P}\in o(H+X^{*})$ so that $\mathbf{P}\in o(G+X^{*})$ by induction. Hence, the same move would win in $G+X$, and $\mathbf{N}\in o(G+X)$. Now suppose the winning move in $H+X$ is to something of the form $\widehat{G^*}+X$ with $\widehat{G^*}$ equal to an option ${G^*}$ of $G$. Then $\mathbf{P}\in o(\widehat{G^*}+X)=o(G^*+X)$, so that $\mathbf{N}\in o(G+X)$. Finally, if the winning move is to an option $H^*$ not equal to an option of $G$, then the other $N-1$ players can move to yield some $H^{*\dots*}+X$ where $H^{*\dots*}=G$. Note that $\mathbf{N}\in o(H^{*\dots*}+X)$ as we started with a winning move in $H+X$. From $H^{*\dots*}=G$ it follows that $\mathbf{N}\in o(G+X)$, as desired.

Suppose $\mathbf{O}_i\in o(H+X)$ for some $i$ with $1\le i\le N-1$. Then $\mathbf{O}_{i-1}\in o(H+X')$ for all options $X'$. By induction, $\mathbf{O}_{i-1}\in o(G+X')$ for all $X'$ as well. Since $\mathbf{O}_i\in o(H+X)$, we also have $\mathbf{O}_{i-1}\in o(\widehat{G'}+X)=o(G'+X)$ for all options $G'$ as well. We have shown $\mathbf{O}_{i-1}\in o\left((G+X)'\right)$ for all options, so $\mathbf{O}_i\in o(G+X)$. \end{proof}

As with three players, replacing options with equal games yields an equal game.

\begin{corollary}\label{cor:nreplacement} 
Suppose that $G\cong \left\{G_i':\,i\in I \right\}$ and $H\cong \left\{H_i':\,i\in I \right\}$ and that $G_i'=H_i'$ for all $i$. Then $G=H$.
\end{corollary}
The proof is identical to that of the three-player case (see Corollary~\ref{cor:replacement}).
\begin{proof}
By symmetry, it suffices to show that $o(H+X)\subseteq o(G+X)$ for all $X$. 
But that follows immediately from Theorem~\ref{thm:reverting-n} since $H$ is revertible to $G$.
\end{proof}

Proposition~\ref{prop:threestar} about $3\cdot*$ generalizes to the $N$-player setting.

\begin{proposition}\label{prop:nstar}
If $N>2$, then $N\cdot*\ne0$.
\end{proposition}
\begin{proof}
In $G\cong\{\{(N-1)\cdot*,0\}\}$, all players except $\mathbf{O}_2$ have winning strategies. There is only one decision; $\mathbf{O}_1$ can win by moving to $0$, in which case $\mathbf{O}_2$ loses. $\mathbf{O}_1$ loses if they move to $(N-1)\cdot*$ instead.

However, in $G+N\cdot*$, players $\mathbf{N}$, $\mathbf{O}_2$ and (if $N>3$) $\mathbf{O}_3$ can cooperate to ensure that $\mathbf{O}_1$ loses. Next can move to $G+(N-1)\cdot*$. Then, for either of $\mathbf{O}_1$'s moves, $\mathbf{O}_2$ can move to $G'+(N-2)\cdot*$. Then $\mathbf{O}_3$ (or $\mathbf{N}$ if $N=3$) can move in $G'$ to $(N-1)\cdot*$, for a total position of $(2N-3)\cdot*$. At this point, $\mathbf{O}_4$ is to move unless $N<5$, in which case it is $\mathbf{O}_{4-N}$. In any case, after $2N-3$ more moves, $\mathbf{O}_1$ loses.
\end{proof}

\subsection{Undetermined and Absorbing Games}
Recall (see Definition~\ref{def:undet}) that a game $G$ is said to be \emph{undetermined} if $o(G)=\emptyset$.

\begin{proposition}\label{prop:sumundet}A sum of undetermined games is undetermined.\end{proposition}
\begin{proof}As with Corollary~\ref{cor:sumundet}, this follows from \hyperref[thm:ngeneratesnplayer]{Next Generation} (Theorem~\ref{thm:ngeneratesnplayer}).\end{proof}
We can make this result more precise.
\begin{defn}
An undetermined game $G$ is said to be $1$-\textit{undetermined}. And, for $k>1$, $G$ is said to be $k$-\textit{undetermined} if $G$ has an option $G^*$ that is $(k-1)$-undetermined.
\end{defn}

In other words, $G$ is $k$-undetermined if there is a directed path of length $k+1$ starting at $G$ through undetermined subpositions. Note that if $G$ is $k$-undetermined, then it is also, \emph{a fortiori}, $j$-undetermined for $1\le j< k$.
\begin{proposition}\label{prop:sumundet2}
If $G$ is $k$-undetermined and $H$ is $m$-undetermined, then $G+H$ is $(k+m)$-undetermined.
\end{proposition}
\begin{proof}
Let $G$ be $k$-undetermined and $H$ be $m$-undetermined.

First, suppose that $k=m=1$. Since $\mathbf O_1\notin o(H)$, we can choose an option $H^*$ with $\mathbf N\notin o(H^*)$. By \hyperref[thm:ngeneratesnplayer]{Next Generation}, we have $o(G+H^*)\subseteq o(G)=\q$. Therefore, $G+H$ is $2$-undetermined. 

If $k>1$, we can choose an undetermined option $G^*$ of $G$. Then $G^*+H$ is $(k-1)+m$-undetermined by induction. Similarly if $m>1$.\end{proof}

The concept of $k$-undetermined games allows us to generalize Corollary~\ref{cor:qgetsnop}.
\begin{lemma}\label{lem:noP}
If $N>2$ and $G$ is $(N-2)$-undetermined, then $\mathbf P\notin o(G+H)$ for all $H$.
\end{lemma}
\begin{proof}
Since $G$ is $(N-2)$-undetermined, we can choose a sequence of $N-2$ subpositions $J_0,J_1,\ldots,J_{N-3}$ where $J_0\cong G$, each $J_{i+1}$ is an option of $J_i$, and each $J_i$ is undetermined. Note that $J_i+H$ is an subposition of $G+H$ reached after $i$ moves.

Suppose, for sake of contradiction, that $\mathbf P\in o(G+H)$ for some particular game $H$. Then $\mathbf P\in o(J_0+H)$, and since $J_0$ is undetermined, \hyperref[thm:ngeneratesnplayer]{Next Generation} tells us that $\mathbf N\in o(H)$ (as otherwise $o(J_0+H)\subseteq o(J_0)=\emptyset$) and $\mathbf P\in o(H)$. 

Since $\mathbf P=\mathbf O_{N-1}$, we have $\mathbf O_{N-2}\in o(J_1+H)$ (if $N>3$). In general, $\mathbf O_{N-1-j}\in o(J_j+H)$ for $0\le j\le N-3$.

Since each $J_j$ is undetermined, $\mathbf O_{N-1-j}\in o(H)$ for $0\le j\le N-3$ as well, in addition to the $\mathbf N\in o(H)$ observed initially. 

Finally, since $J_{N-3}$ is undetermined, $\mathbf O_1\notin o(J_{N-3})$. Thus, we can choose an option $J_{N-3}^*$ such that $\mathbf N\notin o\left(J_{N-3}^*\right)$. But since $\mathbf O_2\in o(J_{N-3}+H)$, we have $\mathbf O_1\in o\left(J_{N-3}^*+H\right)$. This forces $\mathbf O_1\in o(H)$. 

Putting all of these conclusions about $H$ together, we have shown that all players have winning strategies in $H$, which is impossible by claim~1 of Proposition~\ref{prop:nimpouts}.
\end{proof}

We use this Lemma to generalize Proposition~\ref{prop:3undetstrong}. Below, $\lfloor x\rfloor$ denotes the greatest integer less than or equal to $x$.
\begin{proposition}\label{prop:Nundetstrong}If $N>2$ and $G$ is $(N-2)$-undetermined, then all options of $\lfloor N/2+2\rfloor\cdot G$ are undetermined.
\end{proposition}
\begin{proof}
Let $G$ be $(N-2)$-undetermined. Suppose, for sake of contradiction, that $G^*$ is an option of $G$ satisfying $o\left(\lfloor N/2+1\rfloor\cdot G+G^*\right)\ne\q$. Since $G$ is $(N-2)$-undetermined, Lemma~\ref{lem:noP} yields $\mathbf{P}\notin o\left(G+\left(\lfloor N/2\rfloor\cdot G+G^*\right)\right)$. Therefore, $\mathbf{O}_j\in o\left(\lfloor N/2+1\rfloor\cdot G+G^*\right)$ for some $j\in\{0,\dots,N-2\}$.

Suppose $j=0$, so $\mathbf{N}\in o\left(\lfloor N/2+1\rfloor\cdot G+G^*\right)$. Then there is a winning move; either $\mathbf{P}\in o\left(\lfloor N/2+1\rfloor\cdot G+G^{**}\right)$ for some $G^{**}$ or $\mathbf{P}\in o\left(\lfloor N/2\rfloor\cdot G+G^{\star}+G^{*}\right)$ for some $G^{\star}$. But both cases are impossible by Lemma~\ref{lem:noP}.

Now suppose $j\ge1$ and $\mathbf O_j\in o\left(\lfloor N/2+1\rfloor\cdot G+G^*\right)$. Since $G$ is undetermined, Proposition~\ref{prop:sumundet} yields that $\lfloor N/2+1\rfloor\cdot G$ is undetermined. By \hyperref[thm:ngeneratesnplayer]{Next Generation}, $\mathbf{N}\in o\left(G^*\right)$. Hence, we can choose a winning move $G^{**}$. Then $\mathbf{O}_{j-1}\in o\left(\lfloor N/2+1\rfloor\cdot G+G^{**}\right)$. If $H$ is a $(j-1)^{\text{st}}$ option of $\lfloor N/2+1\rfloor\cdot G+G^{**}$, then $\mathbf{N}\in o(H)$. 

In particular, by taking the option $G^{**}$ of each component in the sum when possible, and the option of $G^*$ of a component when necessary, we may choose an option $H$ so that it has at least two components of $G$. For example, if $N=6$ and $j=N-2=4$, then $\mathbf O_3\in 4\cdot G+G^{**}$, $\mathbf O_1\in 3\cdot G+2\cdot G^{**}$, and $\mathbf N\in 2\cdot G+G^*+2\cdot G^{**}$.

Since $\mathbf N\in o(H)$, there is an option $H^*$ with $\mathbf P\in o\left(H^*\right)$. Since our chosen $H$ has at least two components of $G$, $H^*$ has at least one. Just as with the $j=0$ case, this is impossible by by Lemma~\ref{lem:noP}.
\end{proof}
We will not use this fact, but the above proof of Proposition~\ref{prop:Nundetstrong} would have worked the same way for any sum of $\lfloor N/2+2\rfloor$ games, each of which is $(N-2)$-undetermined. 

Recall (see Definition~\ref{def:absorb}) that a game $G$ is said to be \emph{absorbing} if $G+H$ is undetermined for all $H$.

We can use Lemma~\ref{lem:noP} to obtain an $N$-player analogue of the \hyperref[thm:blackholegen]{Absorbing Game Construction}.

\begin{defn}
For $k>1$, a game $G$ is said to be \textit{strongly }$k$\textit{-undetermined} if $G\ncong 0$ and and all options of $G$ are $(k-1)$-undetermined.
\end{defn}
Equivalently, we could replace ``$G\ncong 0$'' with ``$G$ is undetermined'' or ``$G$ is $k$-undetermined'' in the above definition.

\begin{proposition}
If $G$ and $H$ are games such that $G$ is strongly $k$-undetermined and $H$ is strongly $m$-undetermined, then $G+H$ is strongly $(k+m)$-undetermined.
\end{proposition}
\begin{proof}
Every option of $G+H$ is of the form $G'+H$ or $G+H'$ and is $(k+m-1)$-undetermined by Proposition~\ref{prop:sumundet2}. \end{proof}

\begin{theorem}[$N$-Player Absorbing Game Construction]\label{thm:nblackholegen}
Suppose that $N>2$ and $G$ is strongly $(N-1)$-undetermined. Then for all games $H$, $G+H$ is undetermined.
\end{theorem}
\begin{proof}
We show by a finite induction on $k$ that $\mathbf O_k\notin o(G+H)$ for any  $H$. 

First, we handle the base case $k=0$. Suppose, for sake of contradiction that $\mathbf N\in o(G+H)$ for some game $H$. Then either $\mathbf P\in o(G^*+H)$ for some $G^*$ or $\mathbf P\in o(G+H^*)$ for some $H^*$. Both cases are impossible by Lemma~\ref{lem:noP}; the former is impossible by since each $G^*$ is $(N-2)$-undetermined, and the latter is impossible since $G$ itself is, \emph{a fortiori}, $(N-2)$-undetermined.

Now suppose $\mathbf O_{k}\in o(G+H)$ for some $k$ with $1\le k\le N-1$. Then since $G$ is undetermined, $\mathbf N\in o(H)$ by \hyperref[thm:ngeneratesnplayer]{Next Generation}. In particular, $H$ has an option $H^*$ and thus $\mathbf O_{k-1} \in o(G+H^*)$, which is impossible by induction.
\end{proof}

Note that Propositions \ref{prop:blackholeabsorbing} and \ref{prop:absorbingunique} hold with the same proofs, so there is a unique equality class of absorbing games which absorbs summands to produce undetermined games.

\begin{corollary}\label{cor:Nundetabs}If $N>2$ and $G$ is $(N-2)$-undetermined, then $\lfloor N/2+2\rfloor\cdot G$ is an absorbing game.\end{corollary}
\begin{proof}Note that $\lfloor N/2+2\rfloor\cdot G$ is $\left((N-2)\lfloor N/2+2\rfloor\right)$-undetermined by Proposition~\ref{prop:sumundet2}. By Proposition~\ref{prop:Nundetstrong}, $\lfloor N/2+2\rfloor\cdot G$ is strongly undetermined. Since $(N-2)\lfloor N/2+2\rfloor\ge N-1$, $\lfloor N/2+2\rfloor\cdot G$ is an absorbing game by the \hyperref[thm:nblackholegen]{Absorbing Game Construction}.\end{proof}

\begin{corollary}\label{cor:megaundetabs}If $N>2$ and $G$ is undetermined, then $(N-2)\lfloor N/2+2\rfloor\cdot G$ is an absorbing game.\end{corollary}
\begin{proof}Note that $G$ is at least $1$-undetermined, so that $(N-2)\cdot G$ is $(N-2)$-undetermined by Proposition~\ref{prop:sumundet2}. Then apply Corollary~\ref{cor:Nundetabs}.\end{proof}

\subsection{\texorpdfstring{$N$}{N}-player Nim}

In this subsection, we apply generalizations of the \hyperref[lem:3periodicity]{periodicity}/\hyperref[lem:3stability]{stability} lemmas to obtain some results about \nim that hold for general $N$.

In this subsection only, it is particularly helpful at times to write complements of outcomes. For example, if $N=5$, then $\overline{\{\mathbf O_2,\mathbf O_3\}}=\{\mathbf N,\mathbf O_1,\mathbf P\}$.

\subsubsection{Periodicity}

\begin{defn}\label{def:nperiodic}
A game $G$ is said to be $N$\textit{-periodic} if the infinite sequence of outcomes $\left(o(G),o(*+G),o(2\cdot*+G),\dots\right)$ has period $N$.\end{defn}

\begin{proposition}[$N$-Player Nim Periodicity]\label{lem:nperiodicity}
Let $G$ be a game. Suppose that all options $G'$ are $N$-periodic, and that $o\left(G\right)=o\left(N\cdot*+G\right)$. Then $G$ is $N$-periodic.
\end{proposition}
The proof is almost identical to that of three-player \hyperref[lem:3periodicity]{Nim Periodicity}.
\begin{proof}
Suppose $k\ge N+1$. Then the options of $k\cdot*+G$ are $\left(k-1\right)\cdot*+G$ and those of the form $k\cdot*+G'$ for options $G'$. These have the same outcomes as $\left(k-N-1\right)\cdot*+G$ (by induction on $k$) and $\left(k-N\right)\cdot*+G'$  (since $G'$ is $N$-periodic), respectively. But those are exactly all of the options of $\left(k-N\right)\cdot*+G$. Since $k\cdot*+G$ has the same outcomes of options, $o\left(k\cdot*+G\right)=o\left(\left(k-N\right)\cdot*+G\right)$, as desired.
\end{proof}

\begin{proposition}\label{prop:2periodic}$*2$ is $N$-periodic.\end{proposition}
\begin{proof}First, note that $0$ is $N$-periodic since $o(N\cdot*)=o(0)=\overline{\{\mathbf N\}}$ as there is only one line of play (or see claim~2 of Proposition~\ref{prop:nimpouts}). In general, $o(i\cdot*)=\overline{\{\mathbf O_i\}}$ for $i<N$.

Since $0$ is $N$-periodic, $*$ is $N$-periodic as well, so that all options of $*2$ are $N$-periodic. Therefore, by \hyperref[lem:nperiodicity]{Nim Periodicity}, it suffices to show that $o(*2+N\cdot*)=o(*2)$. 

We show by induction on $i$ that for $0\le i\le N-3$, $o(*2+i\cdot*)=\overline{\{\mathbf O_{i+1},\mathbf O_{i+2}\}}$. Note that if $N=2$, then this is vacuously true. If $i=0$, then Next can choose which of $\mathbf O_1$ and $\mathbf O_2$ will lose; $o(*2)=\overline{\{\mathbf O_{1},\mathbf O_{2}\}}$. 

For $i>0$, note that Next can move to $*2+(i-1)\cdot*$, which has outcome $\overline{\{\mathbf O_{i},\mathbf O_{i+1}\}}$, by induction. Since $i\le N-3$, we have $i+2\le N-1$, so that $\mathbf O_{i+1},\mathbf O_{i+2}\notin o(*2+i\cdot*)$. The other options of $*2+i\cdot*$ are $i\cdot*$ and $(i+1)\cdot*$, which have outcomes $\overline{\{\mathbf O_i\}}$ and $\overline{\{\mathbf O_{i+1}\}}$, respectively. Aside from the three players Next, $\mathbf O_{i+1}$, and $\mathbf O_{i+2}$, all other players have winning strategies. And since $i\ne N-1$, the move to $i\cdot*$ is a winning move (i.e. $\mathbf P\in o(i\cdot *)$), and so $\mathbf N\in o(*2+i\cdot*)$. Thus, $o(*2+i\cdot*)=\overline{\{\mathbf O_{i+1},\mathbf O_{i+2}\}}$, as desired.

Similarly, $o(*2+(N-2)\cdot*)=\overline{\{\mathbf P\}}$ since the options have outcomes $\overline{\{\mathbf O_{N-2}\}}$ (which includes $\mathbf P$), $\overline{\{\mathbf P\}}$, and (if $N>2$) $\overline{\{\mathbf O_{N-2},\mathbf P\}}$, respectively.

Then $o(*2+(N-1)\cdot *)=\overline{\{\mathbf O_1\}}$ since the options have outcomes  $\overline{\{\mathbf P\}}$, $\overline{\{\mathbf N\}}$ (which includes $\mathbf P$), and $\overline{\{\mathbf P\}}$, respectively.

Finally, we consider $o(*2+N\cdot*)$. The options have outcomes $\overline{\{\mathbf N\}}$, $\overline{\{\mathbf O_1\}}$, and $\overline{\{\mathbf O_1\}}$, respectively. Thus, if $N>2$, $o(*2+N\cdot*)=\overline{\{\mathbf O_{1},\mathbf O_{2}\}}=o(*2)$. (If $N=2$, $o(*2+N\cdot*)=\overline{\{\mathbf O_{1}\}}=\{\mathbf N\}=o(*2)$.)
\end{proof}

\begin{corollary}The periodic sequence of outcomes $o(*2+i\cdot*)$ is \[\left(\overline{\{\mathbf O_{1},\mathbf O_{2}\}}, \overline{\{\mathbf O_{2},\mathbf O_{3}\}},\ldots,\overline{\{\mathbf O_{N-2},\mathbf P\}}, \overline{\{\mathbf P\}}, \overline{\{\mathbf O_1\}},\dots\right)\text{.}\]\end{corollary}
\begin{proof}This follows from the proof of Proposition~\ref{prop:2periodic} above.\end{proof}

\subsubsection{Stability}

\begin{proposition}[$N$-player Nim Stability]\label{lem:nstability}
Let $G$ be a game. Suppose that for all subpositions $H$ of $G$ (including $G$ itself) $o\left(H+*(N+1)\right)=o\left(H+*N\right)$. Then $o\left(G+*m\right)=o\left(G+*N\right)$ for all $m\ge N$.
\end{proposition}
\begin{proof}
The proof is identical to that of \hyperref[lem:3stability]{Three-Player Nim Stability} (Lemma~\ref{lem:3stability}), except that $4$ should be replaced with $k+1$. 
\end{proof}
In fact, the proof would work with $*N$ replaced with any $*k$. However, when $G$ is a ``small'' \nim position, it seems that we first observe stability at $*N$.

We can use \hyperref[lem:nstability]{Nim Stability} to calculate the outcomes of all \nim positions with at most two heaps, assuming $N>2$. The results are collected in Theorem~\ref{thm:twoheaps} and can be summarized as follows: Players other than $\text{Next}$ and $\mathbf O_1$ do not have a winning strategy unless the game is guaranteed to end before their turn. Ignoring the case of the empty game $0$, $\text{Next}$ fails to have a winning strategy exactly when there are two heaps of size at least $N-1$, and $\mathbf O_1$ fails to have a winning strategy when there is only one heap or if the sizes of the two heaps are at least $N-1$ and $N$, respectively. 

\begin{proposition}\label{prop:oneheap}For $i\ge1$, $o(*i)=\overline{\{\mathbf O_1,\dots,\mathbf O_{\min\{i,N-1\}}\}}$.
\end{proposition}
\begin{proof}
We use induction on $i$ for $1\le i<N$.

If $i=1$, then there is only one line of play, and the outcome is certainly $\overline{\{\mathbf O_1\}}$. For $1<i<N$, the options are $0$ and all $*j$ with $1\le j\le i-1$. The former has outcome $\overline{\{\mathbf N\}}$, and the latter have outcomes of the form $\overline{\{\mathbf O_1,\dots,\mathbf O_j\}}$. Since $\mathbf P\in o(0)$, $\mathbf N\in o(*i)$. Since $\mathbf N\notin o(0)$, $\mathbf O_1\notin o(*i)$. And since $\mathbf O_1,\dots,\mathbf O_{i-1}\notin o(*(i-1))$, $\mathbf O_2,\dots,\mathbf O_i\notin o(*i)$. Thus, $o(*i)=\overline{\{\mathbf O_1,\dots,\mathbf O_{\min\{i,N-1\}}\}}$. 

Similarly, $o(*N)=\overline{\{\mathbf O_1,\dots,\mathbf O_{N-1}\}}=\{\mathbf N\}$ by the same argument as for $*(N-1)$ except that it doesn't matter that $\mathbf P\notin o\left(*(N-1)\right)$. By the same argument, or by the argument in the proof of \hyperref[lem:nstability]{Nim Stability}, this extends to all larger heaps as well.
\end{proof}

\begin{lemma}\label{lem:twosmalls}
If $1\le i\le \min\{j,N-2\}$ then $o(*i+*j)=\overline{\{\mathbf O_2,\dots,\mathbf O_{\min\{i+j,N-1\}}\}}$.
\end{lemma}
\begin{proof}
Note that the hypothesis forces $N>2$. First, $o\left(*+*\right)=\overline{\left\{\mathbf O_{2}\right\}}$ since there is only one line of play.

Now suppose $i+j>2$. Then the options of $*i+*j$ are $*i$, $*j$, and all smaller sums of two nonempty heaps of the forms $(*i)'+*j$ and $*i+(*j)'$.

Note that $\mathbf O_1\notin o(*i)$ by Proposition~\ref{prop:oneheap}, so that $\mathbf{O}_2\notin o(*i+*j)$. 

Since $i+j>2$, either $1\le i\le j-1$ or $1\le j-1\le i$. Thus, by induction, $o\left(*i+*(j-1)\right)=\overline{\{\mathbf O_2,\dots,\mathbf O_{\min\{i+j-1,N-1\}}\}}$. Hence, $\mathbf O_3,\ldots,\mathbf O_{\min\{i+j,N-1\}}\notin o(*i+*j)$. 

Therefore, $o(*i+*j)\subseteq\overline{\{\mathbf O_2,\dots,\mathbf O_{\min\{i+j,N-1\}}\}}$.

Since $i\le N-2$, $\mathbf P=\mathbf O_{N-1}\in o(*i)$ by Proposition~\ref{prop:oneheap}, so that $\mathbf N\in o(*i+*j)$.

Finally, note that $\mathbf N$ and (if $i+j\le N-1$) $\mathbf O_{\min\{i+j,N-1\}},\ldots,\mathbf O_{N-1}$ are included in the outcome of each option (either by Proposition~\ref{prop:oneheap} or induction, as applicable). Thus, $\mathbf O_1$ and (if $i+j<N-1$) $\mathbf O_{\min\{i+j+1,N-1\}},\ldots,\mathbf O_{N-1}$ are elements of $o(*i+*j)$. In other words, $\overline{\{\mathbf O_2,\dots,\mathbf O_{\min\{i+j,N-1\}}\}}\subseteq o(*i+*j)$.

Since we have shown both inclusions, $o(*i+*j)=\overline{\{\mathbf O_2,\dots,\mathbf O_{\min\{i+j,N-1\}}\}}$.\end{proof}

\begin{lemma}\label{lem:almostundet}
$o\left(2\cdot*(N-1)\right)=\{\mathbf O_1\}$.
\end{lemma}
\begin{proof}
If $N=2$, then $o(2\cdot*)=\{\mathbf P\}=\{\mathbf O_1\}$ since there is only one line of play. 

Now assume $N>2$. Then the options of $2\cdot*\left(N-1\right)$ are $*(N-1)$ (with outcome $\{\mathbf N\}$ by Proposition~\ref{prop:oneheap}), and those with two heaps (all of which have outcome $\overline{\left\{\mathbf O_{2}\cdots \mathbf O_{N-1}\right\}}=\left\{\mathbf N,\mathbf O_1\right\}$ by Lemma~\ref{lem:twosmalls}). Since no option's outcome includes $\mathbf P$ (as $N>2$) and the only element of an outcome shared by all options is $\mathbf N$, $o\left(2\cdot*(N-1)\right)=\{\mathbf O_1\}$.
\end{proof}

Recall (see Definition~\ref{def:undet}) that a game $G$ is said to be \textit{undetermined} if $o(G)=\emptyset$.

\begin{lemma}\label{lem:fewundets}
If $N>2$, then for $N-1\le i\le N+1$ and $N\le j\le N+1$, $*i+*j$ is undetermined.
\end{lemma}
\begin{proof}
Firstly, we consider $*(N-1)+*N$. It has two single-heap options each with outcome $\{\mathbf N\}$ by Proposition~\ref{prop:oneheap}, the option $2\cdot*(N-1)$ with outcome $\{\mathbf O_1\}$ by Lemma~\ref{lem:almostundet}, and various options (since  $N>2$) covered by Lemma~\ref{lem:twosmalls} with outcome $\left\{\mathbf N,\mathbf O_1\right\}$. Since $N>2$, $\mathbf O_1\ne\mathbf P$, so that $o\left(*(N-1)+*N\right)=\emptyset$.

Similarly,  $*N+*N$ and $*(N-1)+*(N+1)$ have all the same sorts of options except that $*N+*N$ lacks ``$2\cdot*(N-1)$'' and both have a new move to the undetermined $*(N-1)+*N$. Thus, $o(*N+*N)=o\left(*(N-1)+*(N+1)\right)=\emptyset$.

$*N+*(N+1)$ is similar to $*N+*N$, but adds the undetermined options $*N+*N$ and $*(N-1)+*(N+1)$. And $*(N+1)+*(N+1)$ is similar to $*N+*(N+1)$, but adds the undetermined option $*N+*(N+1)$.
\end{proof}

\begin{theorem}\label{thm:twoheaps}If $N>2$, then for $i\le j$, \[o(*i+*j)=\begin{cases}
\overline{\left\{ \mathbf{N}\right\} } & \text{ if }i=j=0\\
\overline{\left\{ \mathbf{O}_{1},\dots,\mathbf{O}_{\min\left\{ j,N-1\right\} }\right\} } & \text{ if }i=0<j \\
\overline{\left\{ \mathbf{O}_{2},\dots,\mathbf{O}_{\min\left\{ i+j,N-1\right\} }\right\} } & \text{ if }1\le i \le N-2\\
\left\{ \mathbf{O}_{1}\right\}  & \text{ if }i=j=N-1\\
\emptyset & \text{ otherwise}
\end{cases}\text{.}\]
\end{theorem}
\begin{proof}
The first four cases follow immediately from Definition~\ref{def:nplayimp}, Proposition~\ref{prop:oneheap}, Lemma~\ref{lem:twosmalls}, and Lemma~\ref{lem:almostundet}, respectively. 

The final claim uses \hyperref[lem:nstability]{Nim Stability} repeatedly.

First, note that $o(0+*N)=o\left(0+*(N+1)\right)$ by Proposition~\ref{prop:oneheap}. Similarly, by Lemma~\ref{lem:twosmalls}, $o(*i+*N)=o\left(*i+*(N+1)\right)$ for $1\le i\le N-2$. Also, by Lemma~\ref{lem:fewundets},  $o\left(*(N-1)+*N\right)=\emptyset=o\left(*(N-1)+*(N+1)\right)$. With all of these results together, we apply \hyperref[lem:nstability]{Nim Stability} to $G=*(N-1)$ to find that $o\left(*(N-1)+*m\right)=\emptyset$ for all $m\ge N$. It remains to check that any sum of two heaps of size at least $N$ is undetermined.

Since $o(*N+*N)=\emptyset=o\left(*N+*(N+1)\right)$ by Lemma~\ref{lem:fewundets}, another application of \hyperref[lem:nstability]{Nim Stability} with $G=*N$ yields $o\left(*N+*m\right)=\emptyset$ for all $m\ge N$. And since $o\left(*N+*(N+1)\right)=\emptyset=o\left(*(N+1)+*(N+1)\right)$, another application with $G=*(N+1)$ yields $o\left(*(N+1)+*m\right)=\emptyset$ for all $m\ge N$.

For $j\ge 1$, $*(N+j)+*N$ and $*(N+j)+*(N+1)$ are both undetermined. Therefore, for $j\ge1$, we can recursively apply \hyperref[lem:nstability]{Nim Stability} to show that ${*(N+j)}+{*m}$ is undetermined for $m\ge N$.
\end{proof}

\begin{corollary}\label{cor:nimabsorb}
If $N>2$, then some \nim positions are absorbing games.\end{corollary}
\begin{proof}
By Theorem~\ref{thm:twoheaps}, $*(N-1)+*(2N-3)$ is $(N-2)$-undetermined. Thus, by Corollary~\ref{cor:Nundetabs}, $\lfloor N/2+2\rfloor\cdot \left(*(N-1)+*(2N-3)\right)$ is an absorbing game. Alternatively, by Lemma~\ref{lem:fewundets} (or Theorem~\ref{thm:twoheaps}), $*(N-1)+*N$ is undetermined. Thus, by Corollary~\ref{cor:megaundetabs}, $\lfloor N/2+2\rfloor(N-2)\cdot\left(*(N-1)+*N\right)$ is absorbing. Alternatively, by Lemma~\ref{lem:fewundets} (or Theorem~\ref{thm:twoheaps}), $2\cdot*N$ is $2$-undetermined. By Proposition~\ref{prop:sumundet2}, $2\lceil N/2-1\rceil\cdot*N$ is $(N-2)$-undetermined. By Corollary~\ref{cor:Nundetabs}, $2\lfloor N/2+2\rfloor\lceil N/2-1\rceil\cdot *N$ is an absorbing game. Thus, $\left((N+4)(N-1)/2\right)\cdot *N$ is absorbing, too.
\end{proof}

\section{Partizan Games}\label{sec:part}
In this section, we discuss the extension of normal play to the $N$-player partizan setting. As in Section~\ref{sec:nimp}, throughout this entire section, $N$ refers to the number of players ($N\ge2$). And we explicitly mention when assuming $N>2$ is necessary.

We begin by setting up notation and machinery to discuss and compare partizan games. In the following subsection, we prove some general results about comparing games. Finally, we consider some specific games generalizing integers in the two-player case.

\subsection{Preliminaries}
Throughout this section, a (partizan) game is now an ordered $N$-tuple of finite sets of partizan games. They are still ``short'' in that they are finite and nonloopy, so that there is a bound on the length of a run. We continue to use $\cong$ to indicate that two games are isomorphic (i.e. their game trees with labeled edges are isomorphic). In general, our definitions and notation parallel the two-player standard as in Section~II.1 of \cite{cgt}. Similar decisions were made in \cite{cinc}, \cite{ncinc}, and \cite{doles}.

The $N$ players are named ``Left'', ``$\text{Center}_{1}$'', ``$\text{Center}_{2}$'', \ldots, ``$\text{Center}_{N-2}$'', and ``Right''. They each make moves in a cyclic fashion, with ``Left'' moving after ``Right'' and before ``$\text{Center}_{1}$'', etc.

We write some games with an extension of the bar notation commonly used for two players. For example, $G\cong\left\{\mathscr{G}^L\mid\mathscr{G}^{C_1}\mid\mathscr{G}^{C_2}\mid\cdots\mid\mathscr{G}^{C_{N-2}}\mid\mathscr{G}^R\right\}$ or \[G\cong\left\{G_1^L,\dots,G_{m_0}^L\mid G_1^{C_1},\dots,G_{m_1}^{C_1}\mid \cdots\mid G_1^R,\dots,G_{m_{N-1}}^R\right\}\text{.}\]

When convenient, we use ``$C_0$'' in place of $L$ and  ``$C_{N-1}$'' in place of $R$. 
In \cite{cinc} and \cite{doles}, which only consider three-player games, our $\text{Center}_1$ is called ``Center'' and our $C_1$ is ``$C$''.

If it is Left's turn to move in $G$, then they choose one of the \textit{Left options} in $\mathscr{G}^L$ to move to. Analogously for the other players.

If there is no option available to a player on their turn, we declare them the unique loser, and the other $N-1$ players all win equally. As in the impartial case, we call this convention \textit{normal play}.

As it should not cause undue confusion, we use $0$ in partizan contexts to denote a game with no options $0\cong\{\,\mid\cdots\mid\,\}$. In general, we interpret impartial games as partizan ones in the natural way; $*\cong\{0\mid0\mid\cdots\mid0\}$, etc. Following \cite{cinc}, we define $1_L\cong\{0\mid\cdots\mid\,\}$, and analogously $1_{C_i}$ and $1_R$. Note that we never use $\Vert$ or similar to denote ``options of options'', as it could be confused with notation like $\{0\mid\,\mid0\}$, which denotes a particular three-player game with no $\text{Center}_1$ options.

\subsubsection{Outcomes, Sums, and Conjugates}

In order to speak of the outcome of a partizan game, it is helpful to examine the impartial games that result when we select a player to make the first move. 

\begin{defn}\label{def:restn}
Given a partizan game $G$, we recursively define $N$ impartial games, the \textit{restrictions} of $G$, as follows. For $0\le i\le N-2$, the $\textit{Center}_i$ \textit{restriction} is $\rest{G}{C_i}\cong\left\{\rest{G_j^{C_i}}{C_{i+1}}\right\}$, where $G_j^{C_i}$ ranges over the options of $G$ for $\textit{Center}_i$. The Right restriction is $\rest{G}{R}\cong\left\{\rest{G_j^{R}}{L}\right\}$, where $G_j^{R}$ ranges over the options of $G$ for Right.
\end{defn}
Recall that if $i=0$ then $\text{Center}_i$ is Left. For examples, $\rest{1_L}{L}=*$, and $\rest{1_L}{C_1}=0$.

The outcome of a two-player game is determined by the pair of impartial outcomes when Left or Right make the first move. We define the outcome in the $N$-player case analogously.

\begin{defn}
The \textit{outcome} $o(G)$ of a partizan game $G$ is the ordered $N$-tuple of the impartial outcomes of the restrictions. \[o(G)=\left(o\left(\rest{G}{L}\right),o\left(\rest{G}{C_1}\right),\ldots,o\left(\rest{G}{C_{N-2}}\right),o\left(\rest{G}{R}\right)\right)\text{.}\]
\end{defn}

For example, $(\n,\p,\o)$ is the three-player outcome of games where Left and only Left can secure the win, no matter which player moves first. And the outcome $(\{\mathbf N,\mathbf P\},\{\mathbf N,\mathbf P\},\{\mathbf N,\mathbf P\},\{\mathbf N,\mathbf P\})$ is the four-player outcome of games where the next player and the previous player can each guarantee that they win (and if they play well, they will certainly win together with some other player), no matter which player moves first.

\begin{proposition}If $N>2$, all $\left(2^N-1\right)^N$ potential outcomes are possible. \end{proposition}
For instance, there are $343$ possible outcomes of a three-player partizan game, and  $50625$ possible outcomes of a four-player partizan game.
\begin{proof}Recall that by claim~4 of Proposition~\ref{prop:nimpouts}, there are $2^N-1$ possible impartial outcomes. Given impartial outcomes $o_1,\dots,o_N$, choose corresponding impartial games with those outcomes $G_1,\dots,G_N$. Then, using $N-1$ pairs of braces in each component, define $G\cong \left\{\, \{\cdots\{G_{1}\}\cdots\}\mid\cdots\mid\{\cdots\{G_{N}\}\cdots\}\,\right\}$. Then $o(G)=\left(o_1,\dots,o_N\right)$ by construction.\end{proof}

Very analogously to the two-player case, we define the three-player analogue of the disjunctive sum.
\begin{defn}Let $G$ and $H$ be games. The \textit{(disjunctive) sum} $G+H$ is defined recursively by \[G+H\cong\left\{G^L+H,G+H^L\mid G^{C_1}+H,G+H^{C_1}\mid\cdots\mid G^R+H,G+H^R\right\}\text{.}\]\end{defn}

The \textit{conjugates} of a game $G$ serve as $N$-player analogues of the two-player negative $-G$, in the sense that the players' roles are interchanged. As these do not serve the role of an additive inverse, they are reminiscent of the general two-player conjugate in \cite{universes}, of which the mis\`{e}re adjoint defined in V.6.3 of \cite{cgt} is a special case.

\begin{defn}\label{def:adjn}
The \textit{first conjugate} of $G$, denoted by $G^{\dagger}$, is defined by \[G^{\dagger}\cong\left\{\left(G^R\right)^{\dagger}\mid\left(G^L\right)^{\dagger}\mid\left(G^{C_1}\right)^{\dagger}\mid\cdots\mid \left(G^{C_{N-2}}\right)^{\dagger}\right\}\text{.}\]
In general, the $k^\text{\textit{th}}$ \textit{conjugate} of $G$ is obtained by taking the first conjugate $k$ times, as in $G^{\dagger\dots\dagger}$. We denote the sum of the first $N-1$ conjugates $G^\dagger+G^{\dagger\dagger}+\cdots+G^{\dagger\dots\dagger}$ by $G^-$.\end{defn}

Note that the $N^{\text{th}}$ conjugate of $G$ is $G$ itself. For examples, $1_{C_1}\cong1_L^\dagger$ and $1_R$ is the $(N-1)^{\text{st}}$ conjugate of $1_L$. If $N=4$, then $1_{C_2}^-=1_{R}+1_{L}+1_{C_1}$.

\begin{proposition}\label{prop:parttriplenon}
For all games $G$, $\mathbf N\notin o\left(\rest{G+G^-}{x}\right)$ for $x=C_0,\ldots C_{N-1}$.
\end{proposition}
\begin{proof}
Analogously to Proposition~\ref{prop:triplenon}, the players who do not move next can mirror all moves in the other $N-1$ terms of the sum.
\end{proof}

Note that there is nothing like an absorbing game in the partizan context.

\begin{proposition}There is no game $G$ with $o(G+H)=\left(\q,\ldots,\q\right)$ for all $H$.\end{proposition}
\begin{proof}Let $G$ be a game, and define $H\cong k\cdot1_L$, where $k$ is greater than the depth of the game tree of $G$. Then Left can win $G+H$, no matter which player moves first, by simply making all of their moves in $H$.\end{proof}

\subsubsection{Inequalities}

In the two-player setting, we define a preorder by (mis\`{e}re or normal play) favorability to Left. With $N$ players, we similarly define preorders by favorability to an individual player. Our definition is reminiscent of, but incompatible with (see Corollary~\ref{cor:cincincomp}), the definitions in \cite{cinc}. The definition parallels the characterization for two-player inequality in Proposition V.6.1 of \cite{cgt}. It was independently defined, in essentially this form, in \cite{doles}.

\begin{defn}\label{def:ineqn}
We write $G\le_L H$ if, for $1\le i\le N$, Left has a winning strategy in $H+X$ moving $i^{\text{th}}$ whenever Left has a winning strategy in $G+X$ moving $i^{\text{th}}$. We also write $G\nleq_LH$ for the negation of $G\leq_LH$. We define $\le_{C_i}$ and $\le_R$ analogously. 
\end{defn}
For example, if $N=3$, $G\le_L H$, and $\mathbf{O}\in o\left(\rest{G+X}{R}\right)$, then $\mathbf{O}\in o\left(\rest{H+X}{R}\right)$.

Many results below are phrased only in terms of $\le_L$, but they apply equally well, \emph{mutatis mutandis}, to the other relations.

\begin{proposition}\label{prop:Lineqn}
$\le_L$ is transitive and reflexive.
\end{proposition}
\begin{proof}
Both properties follow immediately from the corresponding properties of implication.\end{proof}

\begin{proposition}\label{prop:Lineqaddn}
If $G\le_L H$ then $G+J\le_L H+J$ for any $J$.
\end{proposition}
\begin{proof}
Suppose that Left can win $(G+J)+X\cong G+(J+X)$ moving $i^{\text{th}}$. Then since $G\le_L H$, Left can win $H+(J+X)\cong (H+J)+X$ moving $i^{\text{th}}$.\end{proof}

\begin{corollary}\label{cor:Lineqbiaddn}
If $G\le_LH$ and $J\le_LK$ then $G+J\le_LH+K$.
\end{corollary}

We define other related notation as well.

\begin{defn}\label{def:Lstrictleqn}
We write $G<_LH$ if $G\le_LH$ but $H\nleq_LG$. We define the strict inequalities for the other players analogously.
\end{defn}

\begin{proposition}\label{prop:Lstricttransn}
If $G<_LH$ and $H\le_LJ$, then $G<_LJ$. Similarly, if $G\le_LH$ and $H<_LJ$, then $G<_LJ$.
\end{proposition}
\begin{proof}We prove only the first claim, as the second claim is similar. By Proposition~\ref{prop:Lineqn}, $G\le_LJ$, so it suffices to show that $J\nleq_LG$. Since $G<_LH$, choose $X$ and a play order so that Left can win $H+X$ moving $i^\text{th}$ but does not have a strategy to win $G+X$ moving $i^\text{th}$. Then, since $H\le_LJ$, Left can win $J+X$ moving $i^\text{th}$ as well. This witnesses that $J\nleq_LG$.\end{proof}

\begin{defn}\label{def:Leqn}
We write $G=_LH$ and say that $G$ and $H$ are \textit{equal for Left} if $G\le_LH$ and $H\le_LG$. We define $=_{C_{i}}$ and $=_R$ analogously.
\end{defn}

\begin{proposition}\label{prop:2playeq}
Note that if this definition is applied to the two-player case ($N=2$), then $G=_LH$ implies $G=_RH$.\end{proposition}
\begin{proof}If Left does not have a winning strategy for $G+X$ under some play order, then Right does under the same play order. Similarly, if Left does have a winning strategy, then Right does not.\end{proof}

As is standard, we define $=$ by the fundamental equivalence (see \cite{cgt}).
\begin{defn}[Partizan Equality]\label{def:pareq}
$G=H$ if $o(G+X)=o(H+X)$ for all $X$.
\end{defn}

\begin{proposition}[Components of Equality]\label{prop:eqn}
$G=H$ exactly when $G=_{C_i}H$ for $0\le i\le N-1$.
\end{proposition}
\begin{proof}
Let $X$ be a game. If $G=_LH$, both $o(G+X)$ and $o(H+X)$ have the same membership (or lack thereof) of $\mathbf{N}$ in the first component, $\mathbf{P}$ in the second component, and an appropriate $\mathbf{O}_{i}$ in each other component. Each other $=_{C_{i}}$ covers another potential element in each component, and $=_R$ covers the last potential element in each component. Thus, $o(G+X)=o(H+X)$. Since $X$ was arbitrary, $G=H$.

The proof of the other direction is immediate.\end{proof}

\begin{corollary}\label{cor:eqinsums}If $G=H$ and $J=K$ then $G+J=H+K$.\end{corollary}
\begin{proof}This follows from \hyperref[prop:eqn]{Components of Equality} and $2N$ applications of Corollary~\ref{cor:Lineqbiaddn}.\end{proof}

\subsection{Inequality Results}

As in the impartial case, the lack of provisos in the definition of a partizan outcome allows us to generalize some two-player normal-play results to the $N$-player setting. Since the player who cannot move is the unique loser, more moves available for Left can never hurt them due to timing issues. 

We first characterize comparisons with the game $0$ and apply this to a study of $G^-$, and then consider more general inequalities.

Many results below are stated only for Left, using $\le_L$ and $=_L$. But, by symmetry, there are corresponding results for each other player.

\subsubsection{Comparisons with Zero}

We begin by confirming some general properties of preordered monoids.
\begin{lemma}\label{lem:signflipn}
Suppose $H+J\le_L0$. If $0\le_LH$, then $J\le_L0$.
\end{lemma}
\begin{proof}
If $0\le_LH$, then by Proposition~\ref{prop:Lineqaddn} we may add $J$ to both sides to obtain $J\le_L H+J\le_L0$.
\end{proof}

\begin{proposition}
Suppose $H+J\le_L0$. If $0<_LH$ then $J<_L0$.
\end{proposition}
\begin{proof}If $0<_L H$ then $0\le_L H$ so that $J\le_L0$ by Lemma~\ref{lem:signflipn}. But if $J=_L0$, then $0\le_LJ$, so that Lemma~\ref{lem:signflipn} yields $H\le_L0$, which would contradict $0<_LH$.\end{proof}

Next, we examine which games $G$ satisfy $0\le_L G$. In the proofs of Lemmas~\ref{lem:centmoven} and \ref{lem:rightmoven} to follow, we take inspiration from Theorem 7 of \cite{zeroalone}.

\begin{lemma}\label{lem:centmoven}
If $N>2$ and a player other than Left or Right has an option in $G$, then $0\nleq_L G$.
\end{lemma}
\begin{proof} Suppose that $\text{Center}_i$ has an option in $G$ for some $i$ with $1\le i\le N-2$. Choose $m$ greater than the depth of the game tree of $G$, and define the games $Y\cong m\cdot\left(1_L\right)^-$ (see Definition~\ref{def:adjn}) and $X$ to be the game with the two options $X^{C_i}\cong0$ and $X^{C_{i+1}}\cong Y$.

Example game trees for the case of $N=3$ are shown below.

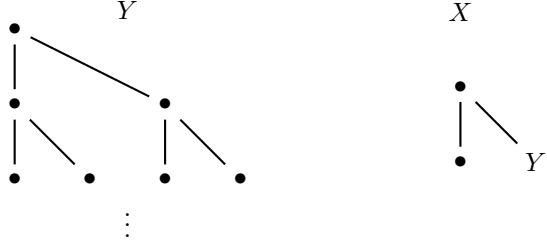
\hfill
\begin{tikzpicture}
\node (meh) at (-0.5,0.25) {$Y$};
\node (root) at (-2,0) {$\bullet$};
\node (c) at (-2,-1) {$\bullet$};
\node (r) at (0,-1) {$\bullet$};
\node (cc) at (-2,-2) {$\bullet$};
\node (cr) at (-1,-2) {$\bullet$};
\node (rc) at (0,-2) {$\bullet$};
\node (rr) at (1,-2) {$\bullet$};
\node (hmm) at (-0.5,-2.5) {$\vdots$};
\draw [thick] (root) -- (c);
\draw [thick] (root) -- (r);
\draw [thick] (c) -- (cc);
\draw [thick] (c) -- (cr);
\draw [thick] (r) -- (rr);
\draw [thick] (r) -- (rc);
\end{tikzpicture}
\hfill
\begin{tikzpicture}
\node (meh) at (0,1) {$X$};
\node (root) at (0,0) {$\bullet$};
\node (c) at (0,-1) {$\bullet$};
\node (r) at (1,-1) {$Y$};
\node (x) at (0,-2) {};
\draw [thick] (root) -- (c);
\draw [thick] (root) -- (r);
\end{tikzpicture}\hfill
\begin{tikzpicture}
\node (meh) at (0,1) {};
\end{tikzpicture}

Since $N>2$, $\text{Center}_{i+1}$ isn't Left. Note that Left can win $X$ with $\text{Center}_i$ moving first, since $\text{Center}_{i+1}$ loses in the only line of play. But in $G+X$, Left cannot guarantee a win with $\text{Center}_i$ moving first.

If $\text{Center}_i$ moves to some $G^{C_i}+X$, then $\text{Center}_{i+1}$ can move to $G^{C_i}+Y$, and then the players other than Left can outlast all of Left's available moves in subpositions of $G$. 

Since Left has a winning strategy in $0+X$ but not $G+X$ when $\text{Center}_i$ moves first, $0\nleq_L G$. 
\end{proof}

\begin{lemma}\label{lem:rightmoven}
If $N>2$ and Right has an option in $G$, then $0\nleq_L G$.
\end{lemma}

\begin{proof} 
Suppose Right has an option in $G$. Define $Y$ as in the proof of Lemma~\ref{lem:centmoven} above, and define the game $X$ to have the two options $X^{C_1}\cong Y$ and $X^R\cong1_L$. 

The case of $N=3$ is illustrated below.

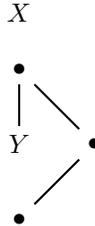
\begin{center}
\begin{tikzpicture}
\node (meh) at (0,0.75) {$X$};
\node (root) at (0,0) {$\bullet$};
\node (c) at (0,-1) {$Y$};
\node (r) at (1,-1) {$\bullet$};
\node (rl) at (0,-2) {$\bullet$};
\draw [thick] (root) -- (c);
\draw [thick] (root) -- (r);
\draw [thick] (r) -- (rl);
\end{tikzpicture}
\end{center}

Note that Left can win $X$ moving second, since $\text{Center}_1$ loses in the only line of play. But in $G+X$, Left cannot guarantee a win moving second.

Right can move to some $G^R+X$. If $G^R$ has no Left option, then Left loses immediately. Otherwise, Left moves to some $G^{RL}+X$, in which case $\text{Center}_1$ can move to $G^{RL}+Y$, and Left will run out of moves first.

Since Left has a winning strategy in $0+X$ but not $G+X$ when moving second, $0\nleq_L G$. 
\end{proof}

\begin{proposition}\label{prop:leftmovesn}
If $G\cong0$ or $G$ only has Left options, then $0\le_L G$.
\end{proposition}
\begin{proof}Suppose Left can win $X$ moving $i^{\text{th}}$. Then Left can win $G+X$ moving $i^{\text{th}}$ by making all moves in $X$, since the other players have no moves available in $G$.\end{proof}

\begin{theorem}[Nonnegativity Rule]\label{thm:posleftn}
If $N>2$, then $0\le_LG$ if and only if $G\cong0$ or $G$ only has Left options.
\end{theorem}
\begin{proof}This is a combination of Lemmas~\ref{lem:centmoven} and \ref{lem:rightmoven}, and Proposition~\ref{prop:leftmovesn}.\end{proof}

\begin{corollary}\label{cor:nosharegood}Suppose that $N>2$, $G\ncong0$, $0\le_{C_i}G$, and $0\le_{C_j}G$. Then $i=j$.\end{corollary}
\begin{proof}By the \hyperref[thm:posleftn]{Nonnegativity Rule}, $G$ only has $\text{Center}_{i}$ options and also only has $\text{Center}_{j}$ options. Since $G\ncong0$,  we must have $i=j$.\end{proof}

\begin{corollary}\label{cor:zeroalonen}
If $N>2$ and $G= 0$, then $G\cong 0$.
\end{corollary}
\begin{proof}
By the \hyperref[prop:eqn]{Components of Equality} and Definition~\ref{def:Leqn}, both $0\le_LG$ and $0\le_RG$. A contrapositive of Corollary~\ref{cor:nosharegood} yields $G\cong 0$.\end{proof}

Note that this dashes any hopes for a good analogue of the two-player result that $1+(-1)=0$. For example, $1_L+\left(1_L\right)^-\ne0$. 
However, we confirm in Theorem~\ref{thm:n=n1n} that non-isomorphic partizan games \emph{can} be equal.

\begin{theorem}[Nonpositivity Rule]$G\le_L0$ if and only if $\mathbf N\notin o(\rest{G}{L})$.
\end{theorem}\label{thm:nonpos}
In other words, if Left doesn't have a winning strategy in $G$ moving first, then $G$ is no better for Left than $0$, in any context. In the two-player case, this corresponds to the well-known fact that $G\le0$ exactly when Right can win $G$ playing second.
\begin{proof}
For one direction, assume $G\le_L0$. Then since Left does not have a winning strategy in $0+0$ moving first, Left does not have a winning strategy in $G+0$ moving first.

For the other direction, assume $\mathbf N\notin o(\rest{G}{L})$ and that Left has a winning strategy in $G+X$ moving $i^{\text{th}}$ for some $i$.

If $2\le i\le N$, then Left can win all relevant $G+X'$ moving $(i-1)^{\text{st}}$. So, by induction on $X$, Left can win all $X'$ moving $(i-1)^{\text{st}}$, and so Left can win $X$ moving $i^{\text{th}}$. Now suppose Left can win $G+X$ moving first. 

If a winning move is to some $G^L+X$, then since $\mathbf N\notin o(\rest{G}{L})$, the players other than Left could play in the $G^L$ component to reach some subposition $
H$ with $\mathbf N\notin o\left(\rest{H}{L}\right)$ and  Left having a winning strategy in $H+X$ moving first. $H\le_L0$ by induction on $G$, so that Left can win $X$ moving first. 

In the other case, a winning move is instead to some $G+X^{L}$. Then Left can win $G+X^{L}$ moving $N^{\text{th}}$. By induction on $X$, Left can win $X^{L}$ moving $N^{\text{th}}$ as well, so that $X^L$ would be a winning move in $X$, too.
\end{proof}

\begin{corollary}\label{cor:semitrin}If $G\cong\left\{(1_L)^-\mid\cdots\mid\,\right\}$ then $G=_L0$.\end{corollary}
\begin{proof}Note that $0\le_LG$ by Proposition~\ref{prop:leftmovesn} (or the \hyperref[thm:posleftn]{Nonnegativity Rule} if $N>2$). But since Left does not win $G$ moving first, $G\le_L0$ by the \hyperref[thm:nonpos]{Nonpositivity Rule}. \end{proof}

\begin{proposition}If $G$ is comparable to $0$ under each $\le_{C_i}$, then one of the following conditions holds, and all cases are possible.
\begin{enumerate}
\item $G=0$.
\item $G<_{C_i}0$ for all $i$.
\item $G=_{C_j}0$ for some $j$ and $G<_{C_i}0$ for $i\ne j$.
\item $G>_{C_j}0$ for some $j$ and $G<_{C_i}0$ for $i\ne j$.\end{enumerate}\end{proposition}
Under Cincotti's convention for $\le_{C_i}$, an additional type of case is possible with precisely two equalities (see Section~4 of \cite{cinc} and \cite{ncinc}).
\begin{proof}
The above cases are all possible with the following examples. Condition 1 is satisfied by $0$. Condition 2 is satisfied by $1_L+(1_L)^-$. Condition 3 is satisfied by a game with only one option in which $\text{Center}_j$ can move to $(1_{C_j})^-$. Condition 4 is satisfied by $1_{C_j}$.

If none of the four conditions were satisfied, then $G\ncong0$ and $G\ge_{C_i}0$ for at least two distinct values of $i$. But this would contradict Corollary~\ref{cor:nosharegood}.
\end{proof}

We also use the two rules for comparisons with $0$ to examine the extent to which $G^-$ is similar to a negative for $G$.

\begin{proposition}\label{prop:trinegn}
For any game $G$, $G+G^-\le_L0$.
\end{proposition}
\begin{proof}
By Proposition~\ref{prop:parttriplenon}, $\mathbf N\notin o(\rest{G+G^-}{L})$. By the \hyperref[thm:nonpos]{Nonpositivity Rule}, $G+G^-\le_L0$.
\end{proof}

To emphasize, $G+G^-\le_{C_i}0$ for any $i$ by symmetry, not just $i=0$. 

\begin{corollary}
If $N>2$ and $G\ncong0$, then $G+G^-<_L0$.
\end{corollary}
\begin{proof}
By Proposition~\ref{prop:trinegn}, $G+G^-\le_L0$. Since $G\ncong0$, $G+G^-$ must have a Right option, so that $0\nleq_LG+G^-$ by the \hyperref[thm:posleftn]{Nonnegativity Rule} (or Lemma~\ref{lem:rightmoven}).
\end{proof}

\begin{proposition}
Suppose that $N>2$, $G\ncong0$, and $0\le_LG$. Then $G^-<_L0$.
\end{proposition}
\begin{proof}By the \hyperref[thm:posleftn]{Nonnegativity Rule}, $G$ has only Left options, and has at least one. Therefore, $G^-$ has no Left options, so that Left loses moving first. Thus, by the \hyperref[thm:nonpos]{Nonpositivity Rule}, $G^-\le_L0$. Since $G^-$ has a non-Left option, we apply the \hyperref[thm:posleftn]{Nonnegativity Rule} to find $0\nleq_LG^-$, and so $G^-<_L0$.\end{proof}
Intuitively, if $G$ never hurts Left, then Left would prefer no change to giving that same benefit to each of the other players. 

Note that we cannot turn this around.
\begin{proposition}
If $N>2$ and $G\ncong0$, then $0\nleq_LG^-$.
\end{proposition}
\begin{proof}This is immediate from the \hyperref[thm:posleftn]{Nonnegativity Rule}.\end{proof}

\subsubsection{General Inequalities}

In the two-player case, there is a recursive characterization of $\le$ (see Thm. II.1.20 in \cite{cgt}). There is a similar partial characterization of $\le_L$ in the $N$-player setting.

\begin{theorem}[Inequality Test]\label{thm:lesstestn}
Let $G\cong\left\{ G^L\mid G^{C_1}\mid\cdots\mid G^{C_{N-1}}\right\}$ and $H\cong\left\{ H^L\mid H^{C_1}\mid\cdots\mid H^{C_{N-1}}\right\}$. Suppose that for any option $G^L$ there is a corresponding $H^L$ such that $G^L\le_LH^L$. Further suppose that for $1\le i\le N-1$ and for all $H^{C_i}$ there is a corresponding $G^{C_i}$ such that $G^{C_i}\le_LH^{C_i}$. Then $G\le_LH$.\end{theorem}
\begin{proof}
Let $X$ be a game. First, suppose that Left can win $G+X$ moving first. If Left can win by moving to some $G^L+X$, then they can win $G^L+X$ moving last. Using the assumption, choose an $H^L$ so that Left can win $H^L+X$ moving last. Then Left can win $H+X$ moving first by moving to $H^L+X$. If Left cannot win $G+X$ by moving in $G$, then Left can win some $G+X^L$ moving last. By induction on $X$, Left can win $H+X^L$ moving last, and Left can still win $H+X$ moving first.

Now suppose that Left can win $G+X$ moving $j^\text{th}$ for some $j$ with $2\le j\le N$. Set $i=N-j+1$. So, with Left moving $(j-1)^\text{st}$, Left can win any $G+X^{C_i}$ and any $G^{C_i}+X$. Then, with Left moving $(j-1)^\text{st}$, Left can win any $H+X^{C_i}$ (by induction on $X$) or $H^{C_i}+X$ (since every $H^{C_i}$ has a corresponding $G^{C_i}$ with $G^{C_i}\le_LH^{C_i}$). But those are all of the $\text{Center}_i$ options of $H+X$, so Left can win $H+X$ moving $j^\text{th}$. 

For all starting players, we have shown that Left having a winning strategy in $G+X$ implies they have one in $H+X$, so $G\le_L H$ by Definition~\ref{def:ineqn}.
\end{proof}

\begin{proposition}\label{prop:parteqreplacementn}
Suppose that $H$ is obtained from $G$ by replacing various options with options for the same player that are equal to the original ones. More precisely, suppose that each of the following holds.
\begin{itemize}
\item $G\cong\left\{\mathscr{G}^{C_0}\mid\mathscr{G}^{C_1}\mid\mathscr{G}^{C_2}\mid\cdots\mid\mathscr{G}^{C_{N-2}}\mid\mathscr{G}^{C_{N-1}}\right\}$
\item $H\cong\left\{\mathscr{H}^{C_0}\mid\mathscr{H}^{C_1}\mid\mathscr{H}^{C_2}\mid\cdots\mid\mathscr{H}^{C_{N-2}}\mid\mathscr{H}^{C_{N-1}}\right\}$
\item For $0\le i\le N-1$,
\begin{itemize}
\item For all $G'\in \mathscr{G}^{C_i}$, there exists $H'\in \mathscr{H}^{C_i}$ satisfying $G'=H'$
\item For all $H'\in \mathscr{H}^{C_i}$, there exists $G'\in \mathscr{G}^{C_i}$ satisfying $G'=H'$
\end{itemize}
\end{itemize}
Then $G=H$.
\end{proposition}
\begin{proof}
By the \hyperref[prop:eqn]{Components of Equality} (Proposition~\ref{prop:eqn}), all of the options of $G$ obey all inequalities with the corresponding options of $H$. For example, if $N>2$, then $H'\le_{C_1}G'$ for all $G'\in\mathscr{G}^{C_2}$ and $H'\in\mathscr{H}^{C_2}$. Applying the \hyperref[thm:lesstestn]{Inequality Test} $2N$ times and using \hyperref[prop:eqn]{Components of Equality} again, results in $G=H$.
\end{proof}

\begin{theorem}[Deleting Dominated Options]\label{thm:deldomn}
Let $G$ be a game with at least two Left options $G^{L_{1}}$ and $G^{L_{2}}$, which satisfy $G^{L_{1}}\le_LG^{L_{2}}$. Suppose that $H$ is obtained from $G$ by removing the Left option $G^{L_{1}}$. Then $H=_LG$.\end{theorem}
\begin{proof}
Since all the options of $H$ are options of $G$, 
the \hyperref[thm:lesstestn]{Inequality Test} yields $H\le_LG$. Since all the options of $G$ are options of $H$ except for $G^{L_{1}}$, and $G^{L_{1}}\le_LG^{L_{2}}$, 
the \hyperref[thm:lesstestn]{Inequality Test} yields $G\le_LH$. The claim follows from Definition~\ref{def:Leqn} for $=_L$.\end{proof}

Note that by Proposition~\ref{prop:2playeq}, this reduces to the standard theorem in the two-player case (see Thm. II.2.4 in \cite{cgt}).

Unfortunately, with more than two players, we only obtain equality for a particular player. In fact, deleting a dominated option need not even preserve the outcome of the game.

\begin{proposition}\label{prop:domnooutcomen}
If $N>2$, there exist games $G$ and $H$ satisfying $G=_LH$ by \hyperref[thm:deldomn]{Deleting Dominated Options}, but $o(G)\ne o(H)$.

\end{proposition}
\begin{proof}
There exist games $J^{L_{1}}$ and $J^{L_{2}}$ such that for the games with no options for players other than Left, $G\cong\left\{ J^{L_{1}},J^{L_{2}}\mid\cdots\mid\,\right\}$ and $H\cong\left\{J^{L_{2}}\mid\cdots\mid\,\right\}$, we have $G=_LH$, but $o\left(G\right)\ne o\left(H\right)$.

Take $J^{L_1}$ to be a game which allows each player to move once in turn, starting with $\text{Center}_1$ and ending with Left, and then $\text{Center}_1$ may move once more (so that $\text{Center}_2$ loses). Take $J^{L_2}$ to be the same, except the last option is another move for Left (so that $\text{Center}_1$ loses) instead of a move for $\text{Center}_1$. 

The game trees for the case of $N=3$ are illustrated below.

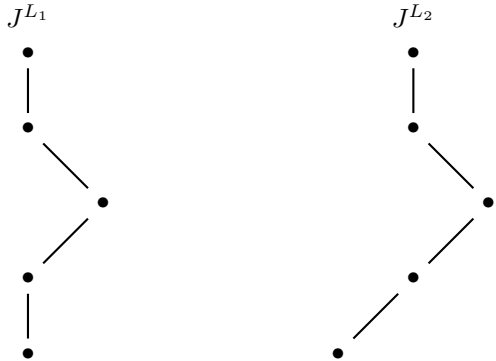
\hfill
\begin{tikzpicture}
\node (rooty) at (0,0.5) {$J^{L_{1}}$};
\node (root) at (0,0) {$\bullet$};
\node (c) at (0,-1) {$\bullet$};
\node (l) at (1,-2) {$\bullet$};
\node (lr) at (0,-3) {$\bullet$};
\node (o) at (0,-4) {$\bullet$};
\draw [thick] (root) -- (c) -- (l) -- (lr) -- (o);
\end{tikzpicture}
\hfill
\begin{tikzpicture}
\node (rooty) at (0,0.5) {$J^{L_{2}}$};
\node (root) at (0,0) {$\bullet$};
\node (c) at (0,-1) {$\bullet$};
\node (l) at (1,-2) {$\bullet$};
\node (lr) at (0,-3) {$\bullet$};
\node (o) at (-1,-4) {$\bullet$};
\draw [thick] (root) -- (c) -- (l) -- (lr) -- (o);
\end{tikzpicture}
\hfill
\begin{tikzpicture}
\node (meh) at (0,1) {};
\end{tikzpicture}

By 
the \hyperref[thm:lesstestn]{Inequality Test}, $1_{C_1}\le_L1_L$. With this as a starting point, we work our way up the game trees to see that multiple applications of 
the \hyperref[thm:lesstestn]{Inequality Test} yield $J^{L_{1}}\le_LJ^{L_{2}}$. Thus, by 
\hyperref[thm:deldomn]{Deleting Dominated Options}, we can delete the dominated option to find $\left\{ J^{L_{1}},J^{L_{2}}\mid\cdots\mid\,\right\}=_L\left\{J^{L_{2}}\mid\cdots\mid\,\right\}$.

To distinguish the outcomes of the two games, we examine the restrictions $\rest{\left\{ J^{L_{1}},J^{L_{2}}\mid\cdots\mid\,\right\}}{L}$ and $\rest{\left\{J^{L_{2}}\mid\cdots\mid\,\right\}}{L}$ to see if $\text{Center}_{2}$ has a winning strategy.

In $\rest{\left\{J^{L_{2}}\mid\cdots\mid\,\right\}}{L}$, the only line of play leads to $\rest{1_{L}}{C_1}$. In this case, $\text{Center}_{1}$ loses, and so $\text{Center}_{2}$ wins. $o\left(\rest{\left\{J^{L_{2}}\mid\cdots\mid\,\right\}}{L}\right)=\left\{\mathbf N,\mathbf O_2,\ldots,\mathbf O_{N-1}\right\}$. 

But in $\rest{\left\{J^{L_{1}},J^{L_{2}}\mid\cdots\mid\,\right\}}{L}$, Left may choose to move to $J^{L_{1}}$ for a line of play ending with $\rest{1_{C_1}}{C_1}$ followed by $\rest{0}{C_2}$. In this case, $\text{Center}_{2}$ loses. Thus, \[o\left(\rest{\left\{J^{L_{1}},J^{L_{2}}\mid\cdots\mid\,\right\}}{L}\right)=\left\{\mathbf N,\mathbf O_3,\ldots,\mathbf O_{N-1}\right\}=o\left(\rest{\left\{J^{L_{2}}\mid\cdots\mid\,\right\}}{L}\right)-\{\mathbf O_2\}\text{.}\]
Since the outcomes of the Left restrictions are not equal, $o(G)\ne o(H)$.
\end{proof}

\begin{theorem}[Bypassing Reversible Options]\label{thm:partrevertingn}
Let $G$ be a game, and suppose that some $N^{\text{th}}$ option satisfies $G^{\widehat{L}\widehat{C_{1}}\cdots\widehat{R}}\le_{L}G$. Put \[G'\cong\left\{G^{\widehat{L}\widehat{C_{1}}\cdots\widehat{R}L},G^{L'}\mid G^C\mid G^R\right\}\text{,}\] where $G^{\widehat{L}\widehat{C_{1}}\cdots\widehat{R}L}$ ranges over all Left options of $G^{\widehat{L}\widehat{C_{1}}\cdots\widehat{R}}$, and $G^{L'}$ ranges over all Left options of $G$ except for $G^{\widehat{L}}$. Then $G'=_LG$.\end{theorem}
\begin{proof}
Suppose Left can win $G+X$ moving first. If they can win by moving to some $G+X^{L}$, then by induction they could win $G'+X^{L}$ moving last, so they can win $G'+X$ moving first. And if Left can win by moving to some $G^{L'}+X$ (not $G^{\widehat{L}}$) then they can do the same in $G'+X$. 

Define $H\cong G^{\widehat{L}\widehat{C_{1}}\cdots\widehat{R}}+X$. The interesting case is if Left can win $G+X$ by moving to $G^{\widehat{L}}+X$. Then the other players can respond by moving to $H$, so Left can win $H$ moving first. 

If Left can win $H$ by moving to some $G^{\widehat{L}\widehat{C_{1}}\cdots\widehat{R}}+X^{L}$, then Left can win $G+X^{L}$ (since $G^{\widehat{L}\widehat{C_{1}}\cdots\widehat{R}}\le_L G$). But then Left could win $G+X$ by moving to $G+X^{L}$, so we're in the former case again. 
And if Left can win $H$ by moving to some $G^{\widehat{L}\widehat{C_{1}}\cdots\widehat{R}L}+X$, they could do the same in $G'+X$.

Now suppose Left has a winning strategy in $G+X$ moving $i^{\text{th}}$ for some $i$ with $2\le i\le N-1$. Then they can win all $G^{C_{N-i+1}}+X$ and $G+X^{C_{N-i+1}}$ moving $(i-1)^\text{st}$. But $G^{C_{N-i+1}}+X$ are some of the $\text{Center}_{N-i+1}$ options of $G'+X$. And, by induction, Left can win the other $\text{Center}_{N-i+1}$ options which have the form $G'+X^{C_{N-i+1}}$. 

Therefore, $G\le_{L}G'$. It remains to show that $G'\le_{L}G$.

Suppose Left can win $G'+X$ moving first. Then there are three cases.

If they win by moving to some $G^{L'}+X$, they can do the same in $G+X$. If they win by moving to $G'+X^{L}$, then by induction they can win by moving to $G+X^{L}$. If they win by moving to $G^{\widehat{L}\widehat{C_{1}}\cdots\widehat{R}L}+X$, this means they could also win $G^{\widehat{L}\widehat{C_{1}}\cdots\widehat{R}}+X$ by moving to $G^{\widehat{L}\widehat{C_{1}}\cdots\widehat{R}L}+X$. Since $G^{\widehat{L}\widehat{C_{1}}\cdots\widehat{R}}\le_LG$, Left can win $G+X$ moving first. 

The other cases where Left can win $G'+X$ moving $i^{\text{th}}$ for some $i$ with $2\le i\le N-1$ are routine.

Since we proved both inequalities, $G'=_{L}G$.\end{proof}

Unfortunately, as with \hyperref[thm:deldomn]{Deleting Dominated Options} (Theorem~\ref{thm:deldomn}), \hyperref[thm:partrevertingn]{Bypassing Reversible Options} need not preserve outcomes (cf. Proposition~\ref{prop:domnooutcomen}).

\begin{proposition}\label{prop:reversenooutcomen}
If $N>2$, then there exists a game $G$ such that $G=_L0$ by \hyperref[thm:partrevertingn]{Bypassing Reversible Options}, but $o(G)\ne o(0)$.
\end{proposition}
\begin{proof}
Take $G\cong\left\{H\mid\cdots\mid\,\right\}$ where $H$ only has two options, both for $\text{Center}_1$; the options are $0$ (to make $\text{Center}_2$ lose) and $1_{C_{2}}+\cdots+1_{C_{N-1}}$. The case of $N=3$ is illustrated below.

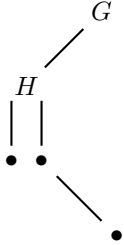
\begin{center}
\begin{tikzpicture}
\node (root) at (0,0) {$G$};
\node (c) at (-1,-1) {$H$};
\node (l) at (-1.2,-2) {$\bullet$};
\node (lr) at (-0.8,-2) {$\bullet$};
\node (o) at (0.2,-3) {$\bullet$};
\draw [thick] (root) -- (c);
\draw [thick] (-0.8,-1.2) -- (lr) -- (o);
\draw [thick] (-1.2,-1.2) -- (l);
\end{tikzpicture}
\end{center}

Note that $0\le_LG$ by Proposition~\ref{prop:leftmovesn}. Thus, by \hyperref[thm:partrevertingn]{Bypassing Reversible Options}, we can bypass the Left option through $0$ to find $G=_L0$. (This could also be verified with the \hyperref[thm:nonpos]{Nonpositivity Rule} as in the proof of Corollary~\ref{cor:semitrin}.)

To distinguish the outcomes of the two games, we examine $\rest{G}{L}$ and $\rest{0}{L}$ to see whether the previous player (Right) has a winning strategy.

In $\rest{0}{L}$, Left loses immediately, so $o\left(\rest{0}{L}\right)=\left\{\mathbf O_1,\ldots,\mathbf O_{N-1}\right\}$. But in $\rest{G}{L}$, Left must move to $\rest{H}{C}$. Then $\text{Center}_1$ may choose to move to $0$ so that $\text{Center}_2$ loses. Thus,  
\[o\left(\rest{G}{L}\right)=\left\{\mathbf O_1,\mathbf O_3,\ldots,\mathbf O_{N-1}\right\}=o\left(\rest{0}{L}\right)-\{\mathbf O_2\}\text{.}\]
Since the outcomes of the Left restrictions are not equal, $o(G)\ne o(0)$.\end{proof}

\subsection{Integers}

Earlier, we defined games with a single move for various players. For example, $1_{C_1}\cong\{\,\mid0\mid\cdots\mid\,\}$. Following \cite{cinc}, we can extend these definitions to define various ``integers''. We can similarly define $2_L\cong\left\{1_L,0\mid\cdots\mid\,\right\}$, $3_L\cong\left\{2_L,1_L,0\mid\cdots\mid\,\right\}$ etc. And analogously for the other players.

In this subsection, as a sort of case study and application of the earlier results, we examine these integers and some related games in detail.

\subsubsection{Comparing Integers}

\begin{proposition}\label{prop:onesn}
If $0\le k<m$, then $k\cdot1_L<_Lm\cdot1_L$.
\end{proposition}
\begin{proof}
By Proposition~\ref{prop:leftmovesn} (or the \hyperref[thm:posleftn]{Nonnegativity Rule} if $N>2$),  $0\le_L1_L$. Then by Proposition~\ref{prop:Lineqaddn}, we can add $j\cdot1_L$ to both sides, to obtain $j\cdot1_L\le_L(j+1)\cdot1_L$ for all $j$. By Proposition~\ref{prop:Lstricttransn}, a strict inequality propagates, so it remains to show that $(j+1)\cdot1_L\nleq_Lj\cdot1_L$ for all $j$. 

Note that by Proposition~\ref{prop:parttriplenon}, Left loses $j\cdot1_L+\left(j\cdot1_L\right)^-$
 moving first. But by counting moves, we see that Left wins $(j+1)\cdot1_L+\left(j\cdot1_L\right)^-$ moving first.
\end{proof}

\begin{proposition}\label{prop:negs}
If $1\le k$ and $1\le i<N$, then $k\cdot1_{C_i}<_L0$.
\end{proposition}
\begin{proof}Since $1_{C_i}$ has no Left options, Left loses immediately when moving first, and so $G\le_L0$ by the \hyperref[thm:nonpos]{Nonpositivity Rule}.

If $N=2$, $k\cdot(-1)\ngeq0$ is well-known. For $N>2$, since $1_{C_i}$ has an option for a player other than Left, the \hyperref[thm:posleftn]{Nonnegativity Rule} yields $G\ngeq_L0$. Putting both inequalities together,  $G<_L0$.\end{proof}

\begin{proposition}\label{prop:incomp}If $N>2$, $1\le i,j\le N-1$, $i\ne j$, and $1\le k,\ell$, then $k_{C_i}\nleq_L \ell_{C_j}$.\end{proposition}
By symmetry, this means that integers for distinct non-Left players are incomparable for Left.
\begin{proof}
Note that Left wins $k_{C_i}+(1_{C_j})^-$ moving first since $\text{Center}_j$ has no move available. But Left loses $\ell_{C_j}+(1_{C_j})^-$ moving first since Left has no move available after each player has moved once (regardless of which move $\text{Center}_j$ makes).
\end{proof}

\begin{lemma}\label{lem:mixedineqn}
If $k<m$, then $m\cdot1_{C_i}<_Lk\cdot1_{C_i}$ for $1\le i\le N-1$.
\end{lemma}
\begin{proof}

First, we prove $\le_L$. If $k=0$, then this follows immediately from 
the \hyperref[thm:lesstestn]{Inequality Test} since $m\cdot1_{C_i}$ has no Left options and $0\cdot1_{C_i}$ has no non-Left options. If $k>0$, then we can add $k\cdot1_{C_i}$ to both sides of $(m-k)\cdot1_{C_i}\le_L0$ by Proposition~\ref{prop:Lineqaddn}.

To show that the inequality is strict, note that Left wins $k\cdot1_{C_i}+\left(m\cdot1_{C_i}\right)^-$ moving first (since $\text{Center}_i$ runs out of moves), but not $m\cdot1_{C_i}+\left(m\cdot1_{C_i}\right)^-$.
\end{proof}

\begin{theorem}\label{thm:n=n1n}For $k\ge 0$, $k_L=k\cdot1_L$.\end{theorem}
\begin{proof}$k=0$ and $k=1$ are trivial, so we assume $k\ge 2$. By definition, $k_L\cong\{0,1_L,2_L,\ldots,(k-1)_L\mid\cdots\mid\,\}$. By induction and Proposition~\ref{prop:parteqreplacementn} allowing the replacement of options with equals, we may assume $k_L=\{0,1_L,2\cdot1_L,\ldots,(k-1)\cdot1_L\mid\cdots\mid\,\}$. 

By Proposition~\ref{prop:onesn}, we see that $(k-1)\cdot1_L$ is the best of these options for Left. Thus, we can repeatedly 
\hyperref[thm:deldomn]{Delete Dominated Options} 
 in $k_L$ to find $k_L=_Lk\cdot1_L$. By 
the \hyperref[prop:eqn]{Components of Equality}, it remains to show that $k_L=_{C_i}k\cdot1_L$ for $1\le i\le N-1$.

First, we show that $k_L\le_{C_i}k\cdot1_L$. 
The \hyperref[thm:lesstestn]{Inequality Test} adapted for $\le_{C_i}$ says that for any $G^{C_i}$ we need a corresponding $H^{C_i}$, but neither game has a $\text{Center}_i$ option. And neither game has options for players other than Left and $\text{Center}_i$, either. So we just need to verify that for all $(k\cdot1_L)^L$, there is a corresponding $k_L^L$ with $k_L^L\le_{C_i}(k\cdot1_L)^L$. But the only option $(k\cdot1_L)^L$ is $(k-1)\cdot1_L$, which is also a Left option of $k_L$ (up to equality), so that $(k-1)\cdot1_L\le_{C_i}(k-1)\cdot1_L$ suffices.

Next, we show that $k\cdot1_L\le_{C_i}k_L$. Similarly to the other direction, we need only handle each Left option of $k_L$. It remains to check that $(k-1)\cdot1_L\le_{C_i}j\cdot1_L$ for $j\le k-1$. This follows immediately from Lemma~\ref{lem:mixedineqn} with players switched. 

Since $k\cdot1_L\le_{C_i}k_L$ and $k_L\le_{C_i}k\cdot1_L$, we have $k_L=_{C_i}k\cdot1_L$ for $1\le i\le N-1$. Since we also verified that $k_L=_{L}k\cdot1_L$, \hyperref[prop:eqn]{Components of Equality} yields $k_L=k\cdot1_L$, as desired.\end{proof}

\subsubsection{Sums of Integers}
\begin{proposition}\label{prop:justones}Evaluating the outcome of a sum of integers (such as $3_{C_2}$) reduces to the case of evaluating the outcome of a game of the form $k_0\cdot1_{C_0}+\cdots+k_{N-1}\cdot1_{C_{N-1}}$ for nonnegative integers $k_0,\ldots,k_{N-1}$.\end{proposition}
\begin{proof}Note that by Theorem~\ref{thm:n=n1n}, a sum of multiple integers for the same player can be reduced to a sum of ones; for example, $k_L+j_L=(k+j)\cdot1_L$. Corollary~\ref{cor:eqinsums} tells us that we can make such replacements throughout a sum of integers for various players. By the definition of \hyperref[def:pareq]{Partizan Equality} (Definition~\ref{def:pareq}), the outcome doesn't change when a game is replaced by an equal one.\end{proof}

\begin{proposition}\label{prop:oneoutcomes}Let $k_0,\ldots,k_{N-1}$ be nonnegative integers with minimum $k_{\text{min}}$, and $i$ be an integer with $0\le i\le N-1$. If $j$ is the least nonnegative integer with $k_{i+j}=k_{\text{min}}$ or $k_{i+j-N}=k_{\text{min}}$, then \[o\left(\rest{k_0\cdot1_{C_0}+\cdots+k_{N-1}\cdot1_{C_{N-1}}}{C_{i}}\right)=\left\{\mathbf O_0,\ldots,\mathbf O_{N-1}\right\}-\{\mathbf O_j\}\text{.}\]
\end{proposition}
\begin{proof}
Suppose that $\text{Center}_i$ moves first in $k_0\cdot1_{C_0}+\cdots+k_{N-1}\cdot1_{C_{N-1}}$. There is only one line of play. After $N*k_{\text{min}}$ moves, all coefficients that were equal to $k_{\text{min}}$ have been reduced to $0$, so that the next player to move (starting with $\text{Center}_i$) with a corresponding $0$ coefficient will lose, and all other players will win.
\end{proof}
\begin{corollary}\label{cor:overwhelm}
In a sum of the form $k_{C_i}+K\cdot\left(1_{C_i}\right)^-$ with $K>k$, $\text{Center}_i$ loses regardless of how they play or which player moves first.
\end{corollary}
\begin{proof}By Propositions~\ref{prop:justones} and \ref{prop:oneoutcomes}, all players other than $\text{Center}_i$ have winning strategies, regardless of which player moves first. But the players other than $\text{Center}_i$ have no choice in their moves, so there is no other way they could play to allow $\text{Center}_i$ to win.\end{proof}

\subsubsection{Games Less than One}

In the two-player context, there are results such as $\{-1\mid\,\}=0$, since a move that benefits Right can't help Left. But with three or more players, things are more subtle.

\begin{theorem}\label{thm:smallbenefitsn}
If $N>2$, then for $1\le i\le N-1$, \[0<_L\cdots<_L\{3_{C_i}\mid\cdots\mid\,\}<_L\{2_{C_i}\mid\cdots\mid\,\}<_L\{1_{C_i}\mid\cdots\mid\,\}<_L1_L\text{.}\]
\end{theorem}
\begin{proof}
There are four claims that must be verified. \begin{enumerate}
\item $0\le_L\{k_{C_i}\mid\cdots\mid\,\}$ for all $k$.
\item $\{k_{C_i}\mid\cdots\mid\,\}\nleq_L0$ for all $k$.
\item $\{m_{C_i}\mid\cdots\mid\,\}\le_L\{k_{C_i}\mid\cdots\mid\,\}$ for $k<m$.
\item $\{k_{C_i}\mid\cdots\mid\,\}\nleq_L\{m_{C_i}\mid\cdots\mid\,\}$ for $k<m$.
\end{enumerate}

Claim~1 follows from the \hyperref[thm:posleftn]{Nonnegativity Rule} (Theorem~\ref{thm:posleftn}).

Claim~2 follows from the fact that Left wins each $\{k_{C_i}\mid\cdots\mid\,\}$ moving first (here we use $N>2$), but does not win $0$ moving first.

Claim~3 follows from Theorem~\ref{thm:n=n1n} to convert the integers into equal sums of $1_{C_i}$, Proposition~\ref{prop:parteqreplacementn} to replace them in the options, Lemma~\ref{lem:mixedineqn} to compare those sums of $1_{C_i}$, and one application of the \hyperref[thm:lesstestn]{Inequality Test} to obtain the desired inequality.

Claim~4 requires more care. To witness the inequality, we define a game $G$ differently in the cases $i=1$ and $i>1$.

If $i=1$, then define $G$ to be the game with only one option in which $\text{Center}_2$ can move to $(m-1)\cdot1_L+(m-1)\cdot1_{C_{2}}+\displaystyle{\sum_{3\le j\le N-1}}m\cdot1_{C_j}$. If $2\le i\le N-1$, then define $G$ to be the game with only one option in which  $\text{Center}_1$ can move to $(m-1)\cdot1_L+(m-1)\cdot1_{C_{1}}+\displaystyle{\sum_{2\le j\le N-1,\,j\ne i}}m\cdot1_{C_j}$. 

Either way, Left has a winning strategy in $\left\{k_{C_i}\mid\cdots\mid\,\right\}+G$ moving first. The only lines of play lead to a position of the form $\widehat{k}_{C_i}+(m-1)\cdot\left(1_{C_i}\right)^-$ for some $\widehat{k}\le k-1$, with Left to move. From here, by Corollary~\ref{cor:overwhelm}, $\text{Center}_i$ will lose even if they play well. However, Left does not have a winning strategy in $\left\{m_{C_i}\mid\cdots\mid\,\right\}+G$ moving first. Assuming $\text{Center}_i$ plays well, play leads to a position equal to $\left((m-1)\cdot1_L\right)+\left((m-1)\cdot1_L\right)^-$ with Left to move, so that Left does not have a winning strategy by Proposition~\ref{prop:parttriplenon}.
\end{proof}

\begin{corollary}\label{cor:cincincomp}
If $N=3$, our $\le_L$ is incompatible with the one defined by Cincotti in Subsection~2.3 of \cite{cinc}.
\end{corollary}
\begin{proof}
Note that $0<_L\{1_C\mid\,\mid\,\}$ by Theorem~\ref{thm:smallbenefitsn}. But under Cincotti's recursive definition, we would have $0\ge_L\{1_C\mid\,\mid\,\}$ since $1_C\ngeq_L0$.
\end{proof}

\section{Conclusions}\label{sec:open}

Throughout this paper, we have seen a variety of generalizations of two-player theorems, and regularity in results that apply for $N>2$. It appears that the normal play convention considered in this paper may be the easiest to investigate for $N$ players, without discarding parts of the game tree (as in \cite{cinc}) or making any assumptions on how players play (as in \cite{doles} and other papers mentioned in I.4 of \cite{cgt}).

That said, when considering combinatorial games with more than two players, there is a world of gaps and fundamental questions that remain unsettled.

For just the case of impartial three-player games, there are other similar play conventions worth considering, each depending on the winner(s) and loser(s) of $0$, analogous to mis\`{e}re play. In \cite{propp}, Propp analyzed the convention in which Previous is the unique winner of $0$. But that still leaves four other non-trivial play conventions that do not seem to have been investigated in the literature. For example, consider the similar convention in which Other is the unique winner of $0$. The 2 conventions other than normal play with a unique loser for $0$ may be particularly difficult to analyze. 

\begin{question}For reach of the remaining four non-trivial play conventions, investigate the sum table analogous to Table~\ref{tbl:addition}.\end{question}

Even when restricting ourselves to normal play, much is still unknown. As noted at the end of Subsection~\ref{ssec:3imp} (just before \ref{ssec:3nim}), we can build many equal impartial games by replacing one absorbing subposition with another. Are there other equal games? For instance, perhaps $3\cdot *=6\cdot *$ in the three-player case.

We might hope to find a recursive test for equality similar to the two-player mis\`{e}re Theorem~V.3.6 from \cite{cgt}. However, no simple translation of that result can handle the three-player case under normal play. A key part of that argument is Lemma~V.3.3, which uses $T\cong\left\{ G_{1}^{-},G_{2}^{-},\ldots,G_{k}^{-},U\right\}$  and then considers $G+T$ and $H+T$. In that case, options such as $G_{i}+T$ and $G+G_{i}^{-}$ have an option of the form $G_{i}+G_{i}^{-}$ with $\mathbf N\notin o\left(G_{i}+G_{i}^{-}\right)$. We can therefore conclude that $\mathbf{O}\notin o(G_{i}+T)$. But without a property such as ``$\mathbf{O}\in o\left(G_{i}+G_{i}^{-}\right)$ for all $G_{i}$'' or ``$\mathbf P\in o\left(G_{i}+G_{i}^{-}\right)$ for some $G_{i}$'', we can't conclude anything further about options of $G+T$ such as $G_{i}+T$, no matter what we know about $G+U$.

\begin{question}For $N>2$, is there a non-isomorphic pair of equal impartial games that do not have an absorbing subposition?\end{question}

Even without more ability to test for equality, it would be nice to complete Proposition~\ref{prop:trebling} about the outcomes of $3\cdot G$. While there were no surprises under normal play for doubling, Claim~9 from Section~4 of \cite{propp} suggests that there could be another obstruction to the outcome of $3\cdot G$ aside from the obstructions that apply to all sums and Proposition~\ref{prop:tripleno} about move mirroring.

\begin{question}Under normal play, is there a game $G$ with three-player outcome $\n$ such that $\mathbf P\in o(3\cdot G)$?\end{question}

Restricting ourselves to the particular game of $N$-player \textsc{Nim} for $N>2$, it is not certain how similar things must be to the three-player case. 

\begin{question}Are all \nim positions $N$-periodic?\end{question}
This need not hold \emph{outside} of normal play, even for $N=3$. For example, in the three-player convention in which Other is the unique winner of $0$, $o(2\cdot*3)=\o$, but $o(3\cdot*+2\cdot*3)=\q$.

In the proof of Corollary~\ref{cor:nimabsorb}, we show that certain very large \nim positions are always absorbing. The author suspects this can be improved considerably.
\begin{question}If $N>2$, do $N\cdot*2$ and $2\cdot*N$ have the absorbing property in the Nim Quotient?\end{question}

If so, then $2(N-1)\cdot*2$ and $N\cdot*N$ are absorbing (for all impartial games) by the \hyperref[thm:nblackholegen]{Absorbing Game Construction} (Theorem~\ref{thm:nblackholegen}).

The above is far from an exhaustive list of avenues for future research. Loopy games, nondisjunctive compounds, other case studies such as Rhombination from \cite{doles}, etc. are all wide open.

\section*{Acknowledgments}The author is grateful for the helpful suggestions from Richard Nowakowski and Richard Biggs. He also wishes to thank Stephen Goodloe and Gregory Puleo for their feedback and support. Finally, Yuki Irie and Nathan Fox were sources of inspiring questions.

\section*{Appendix}\label{sec:appendix}

In this appendix, we follow the compact notation for impartial games used in V.2 of \cite{cgt} for analyzing two-player mis\`{e}re play. 

\newpage

Specifically, 

\begin{itemize}
	\item $*GHJ$ denotes the game $\left\{*G,*H,*J\right\}$. For example, note that $*21$ denotes $\left\{*2,*\right\}$, rather than a nim-heap of size twenty-one.
	\item $*G_{\#}$ denotes the game $\left\{*G\right\}$.
	\item $*G_{H}$ denotes $*G+*H$.
\end{itemize}

Some examples of this notation from \cite{cgt} are as follows.
\begin{itemize}
	\item $*2_{\#}320\cong\left\{\left\{*2\right\},*3,*2,0\right\}$
	\item $*2_{\#\#2}\cong\left\{\left\{*2\right\}\right\}+*2$
	\item $*2_{2\#\#}\cong\left\{\left\{*2+*2\right\}\right\}$
\end{itemize}

The following tables are all for impartial games with three players.

\setlength{\tabcolsep}{4pt}
\begin{table}[h!]
\centering
\resizebox{\linewidth}{!}{
\begin{tabular}{rccccccc}
  &  $\q$  &  $\n$  &  $\o$  &  $\p$ & $\no$ & $\op$ & $\pn$\\
$\q$ & $*2_{\#}2$ & $\times$ & $\times$ & $\times$ & $\times$ & $\times$ & $\times$ \\
 $\n$  &  $*3$ &  $*(1_{\#}1)20$  &  $*2$ &  $\times$ & $*(1_{\#}1)(20)0$ & $\times$ & $\times$\\
 $\o$ & $*2_2$ & $\times$ & $*2_{\#}$ & $\times$ & $\times$ & $\times$ & $\times$ \\
 $\p$ & $*2_21_{\#}$ & $\times$ & $\times$ & $\times$ & $\times$ & $\times$ & $\times$ \\
 $\no$  &  $*2_{1}21$  &  $*21$ & $*2_{1}1$  & $*(2_{1}0)_{\#}$ & $*2_{1}$ & $*(1_{\#}1)1$ & $*1_{\#}$\\
 $\op$  &  $*(2_{1})_{\#}$  &  $\times$ &  $*(2_{1}1)_{\#}$ & $*2_{1}1_{\#}$ & $\times$ & $0$ & $\times$\\
 $\pn$  &  $*(2_{1})_{\#\#}$  &  $*2_{1}0$  &  $*(1_{\#}2_{1})_{\#}$ &  $\times$  & $*$ & $\times$ & $\times$
\end{tabular}
}
\caption{Doubling examples --- The row is $o(G)$ and the column is $o(G+G)$.}
\label{tbl:doublesums}
\end{table}
\setlength{\tabcolsep}{6pt}

\begin{table}[h!]
\centering
\resizebox{\linewidth}{!}{
\begin{tabular}{rcccc}
 & $\q$ & $\o$ & $\p$ & $\op$ \\
$\q$ & $*2_23$  & $\times$ & $\times$  & $\times$ \\
$\n$ & $*2$  & $*((1_{\#}1)1)20$  & $\text{?}$ & $\text{?}$ \\
$\o$ & $*2_{\#}$  & $*1_{\#}2$  & $\times$ & $\times$ \\
$\p$ & $*2_{\#}1_{\#}$  & $\times$  & $\times$ & $\times$ \\
$\no$ & $*21$  & $*(2_10)_{\#}$ & $*(21)_{\#\#\#}$ & $*1_{\#}$ \\
$\op$ & $*2_11_{\#}$ & $*(2_10)_{\#\#}$ & $*(21)_{\#}$ & $0$ \\
$\pn$ & $*2_{\#}1_{\#}0$ & $*2_10$  & $*1_{\#}0$  & $*$ 
\end{tabular}
}
\caption{Trebling Examples --- The row is $o(G)$ and the column is $o(3\cdot G)$.}
\label{tbl:treblesums}
\end{table}

\newpage
\setlength{\tabcolsep}{3.3pt}
\begin{table}[!ht]
\centering
\begin{tabular}{l@{\hspace{1.2\tabcolsep}}l@{\hspace{1\tabcolsep}}l@{\hspace{1.6\tabcolsep}}l@{\hspace{1.2\tabcolsep}}l@{\hspace{1\tabcolsep}}l}
$\q+\q=\q$ & $*2_{\#}2$ & $*2_{\#}2$&$\o+\pn=\n$ & $*2_{2}$ & $*$\\
$\q+\n=\q$ & $*2_{\#}2$ & $*2$&$\o+\pn=\p$ & $*2_{\#}$ & $*$\\
$\q+\n=\n$ & $*2_{\#\#\#2}1_{\#}$ & $*2$&$\o+\pn=\pn$ & $*1_{\#}2$ & $*$\\
$\q+\o=\q$ & $*2_{\#}2$ & $*2_{\#}$&$\p+\p=\q$ & $*2_{\#\#}$ & $*2_{\#\#}$\\
$\q+\p=\q$ & $*2_{\#}2$ & $*2_{\#\#}$&$\p+\no=\q$ & $*2_{\#\#}$ & $*1_{\#}1$\\
$\q+\no=\q$ & $*2_{2}2$ & $*1_{\#}1$&$\p+\no=\n$ & $*2_{2}1_{\#}$ & $*(1_{\#}0)_{\#}$\\
$\q+\no=\n$ & $*2_{2}2$ & $*1_{\#}$&$\p+\no=\o$ & $*2_{\#\#}$ & $*1_{\#}$\\
$\q+\no=\o$ & $*2_{\#}2$ & $*1_{\#}1$&$\p+\no=\no$ & $*2_{2}1_{\#}$ & $*1_{\#}$\\
$\q+\no=\no$ & $*2_{\#}2$ & $*1_{\#}$&$\p+\op=\q$ & $*2_{\#\#}$ & $*2_{1\#}$\\
$\q+\op=\q$ & $*2_{\#}2$ & $0$&$\p+\op=\p$ & $*2_{\#\#}$ & $0$\\
$\q+\pn=\q$ & $*2_{2}2$ & $*$&$\p+\pn=\q$ & $*2_{2\#}$ & $*2_{1\#\#}$\\
$\q+\pn=\n$ & $*2_{\#}2$ & $*$&$\p+\pn=\n$ & $*2_{\#\#}$ & $*$\\
$\n+\n=\q$ & $*2$ & $*20$&$\no+\no=\q$ & $*31$ & $*(1_{\#}0)_{\#}$\\
$\n+\n=\n$ & $*2$ & $*1_{\#}20$&$\no+\no=\n$ & $*1_{\#}$ & $*1_{\#}1$\\
$\n+\n=\o$ & $*2$ & $*2$&$\no+\no=\o$ & $*2_{1}1$ & $*2_{1}1$\\
$\n+\n=\no$ & $*2$ & $*1_{\#}10$&$\no+\no=\p$ & $*1_{\#}$ & $*(1_{\#}0)_{\#}$\\
$\n+\o=\q$ & $*2$ & $*2_{2}$&$\no+\no=\no$ & $*1_{\#}1$ & $*1_{\#}1$\\
$\n+\o=\n$ & $*2$ & $*2_{\#}$&$\no+\no=\op$ & $*(1_{\#}1)1$ & $*(1_{\#}1)1$\\
$\n+\p=\q$ & $*2_{\#\#\#}$ & $*2_{\#\#}$&$\no+\no=\pn$ & $*1_{\#}$ & $*1_{\#}$\\
$\n+\p=\n$ & $*2$ & $*2_{\#\#}$&$\no+\op=\q$ & $*1_{\#}1$ & $*2_{1\#}$\\
$\n+\no=\q$ & $*2_{\#2}$ & $*1_{\#}$&$\no+\op=\n$ & $*1_{\#}$ & $*(1_{\#}0)_{\#\#}$\\
$\n+\no=\n$ & $*1_{\#}10$ & $*1_{\#}$&$\no+\op=\o$ & $*1_{\#}$ & $*2_{1\#}$\\
$\n+\no=\o$ & $*2$ & $*2_{1}1$&$\no+\op=\no$ & $*1_{\#}$ & $0$\\
$\n+\no=\p$ & $*2_{\#\#\#}$ & $*1_{\#}$&$\no+\pn=\q$ & $*2_{1}1$ & $*1_{\#}0$\\
$\n+\no=\no$ & $*20$ & $*1_{\#}$&$\no+\pn=\n$ & $*(1_{\#}0)_{\#}$ & $*1_{\#}0$\\
$\n+\no=\op$ & $*2$ & $*(1_{\#}1)1$&$\no+\pn=\o$ & $*(1_{\#}0)_{\#}$ & $*$\\
$\n+\no=\pn$ & $*2$ & $*1_{\#}$&$\no+\pn=\p$ & $*1_{\#}$ & $*2_{1\#\#}$\\
$\n+\op=\q$ & $*2_{\#\#\#}$ & $*2_{1\#}$&$\no+\pn=\no$ & $*2_{1}1$ & $*$\\
$\n+\op=\n$ & $*2$ & $0$&$\no+\pn=\op$ & $*1_{\#}$ & $*$\\
$\n+\pn=\q$ & $*2_{\#\#2}$ & $*$&$\no+\pn=\pn$ & $*1_{\#}1$ & $*$\\
$\n+\pn=\n$ & $*1_{\#}10$ & $*$&$\op+\op=\q$ & $*2_{1\#}$ & $*2_{1\#}$\\
$\n+\pn=\o$ & $*2_{\#\#\#}$ & $*$&$\op+\op=\o$ & $*1_{\#\#}$ & $*(1_{\#}0)_{\#\#}$\\
$\n+\pn=\no$ & $*2$ & $*$&$\op+\op=\p$ & $*1_{\#\#}$ & $*2_{1\#}$\\
$\o+\o=\q$ & $*2_{\#}$ & $*2_{2}$&$\op+\op=\op$ & $0$ & $0$\\
$\o+\o=\o$ & $*2_{\#}$ & $*2_{\#}$&$\op+\pn=\q$ & $*2_{1\#}$ & $*2_{1\#\#}$\\
$\o+\p=\q$ & $*2_{\#}$ & $*2_{\#\#}$&$\op+\pn=\n$ & $*2_{1\#}$ & $*$\\
$\o+\no=\q$ & $*2_{2}$ & $*2_{1}$&$\op+\pn=\p$ & $*(1_{\#}0)_{\#\#}$ & $*1_{\#\#\#}$\\
$\o+\no=\n$ & $*2_{\#}$ & $*1_{\#}$&$\op+\pn=\pn$ & $0$ & $*$\\
$\o+\no=\o$ & $*2_{2}$ & $*1_{\#}1$&$\pn+\pn=\q$ & $*2_{1\#\#}$ & $*2_{1\#\#}$\\
$\o+\no=\no$ & $*1_{\#}2$ & $*1_{\#}$&$\pn+\pn=\n$ & $*$ & $*1_{\#}0$\\
$\o+\op=\q$ & $*1_{\#}2$ & $*2_{1\#}$&$\pn+\pn=\o$ & $*1_{\#\#\#}$ & $*2_{1\#\#}$\\
$\o+\op=\o$ & $*2_{\#}$ & $0$&$\pn+\pn=\no$ & $*$ & $*$\\
$\o+\pn=\q$ & $*2_{\#}$ & $*1_{\#}0$\\
\end{tabular}
\caption{Examples of pairs of impartial games witnessing all outcomes of sums}
\label{tbl:allsums}
\vspace{-2.1pt}
\end{table}
\setlength{\tabcolsep}{6pt}

\clearpage 

\end{document}